\documentclass[11pt]{article}

\usepackage{mathrsfs}
\usepackage{nicefrac}
\usepackage[T1]{fontenc}
\usepackage{latexsym,amssymb,amsmath,amsfonts,amsthm}
\usepackage{graphicx}
\usepackage{caption}
\usepackage[utf8]{inputenc}
\usepackage{authblk}
\usepackage{bm}
\usepackage{accents}
\usepackage{booktabs}
\usepackage[pdftex,colorlinks=true,citecolor=blue,linkcolor=blue]{hyperref}

\newtheorem{theorem}{Theorem}[section]
\newtheorem{lemma}{Lemma}[section]
\newtheorem{corollary}{Corollary}[section]
\newtheorem{remark}{Remark}[section]
\newtheorem{example}{Example}[section]

\DeclareMathOperator*{\rank}{rank}
\DeclareMathOperator*{\diag}{diag}
\DeclareMathOperator*{\tr}{tr}

\allowdisplaybreaks[4]

\begin{document}
	
\title{On the perturbation of an $L^{2}$-orthogonal projection}
\author{
Xuefeng Xu%
\thanks{Department of Mathematics, Purdue University, West Lafayette, IN 47907, USA (\texttt{xuxuefeng@lsec.cc.ac.cn; xu1412@purdue.edu}).}
}
%
\date{\today}
\maketitle

\begin{abstract}

The $L^{2}$-orthogonal projection onto a subspace is an important mathematical tool, which has been widely applied in many fields such as linear least squares problems, eigenvalue problems, ill-posed problems, and randomized algorithms. In some numerical applications, the entries of a matrix will seldom be known exactly, so it is necessary to develop some bounds to characterize the effects of the uncertainties caused by matrix perturbation. In this paper, we establish new perturbation bounds for the $L^{2}$-orthogonal projection onto the column space of a matrix, which involve upper (lower) bounds and combined upper (lower) bounds. The new bounds contain some sharper counterparts of the existing ones. Numerical examples are also given to illustrate our theoretical results.

\end{abstract}
	
	
\medskip
	

\section{Introduction}

The $L^{2}$-orthogonal projection onto a subspace is an important geometric construction in finite-dimensional spaces, which has been applied in many fields such as linear least squares problems, eigenvalue (singular value) problems, ill-posed problems, and randomized algorithms (see, e.g.,~\cite{Nelson1987,Fierro1995,Fierro1996,Jia1999,Morigi2006,Morigi2007,Grcar2010,Hochstenbach2010,Coakley2011,Drineas2011,Golub2013,Betcke2017}). However, in some numerical applications, the entries of a matrix will seldom be known exactly. Thus, it is necessary to establish some bounds to characterize the effects arising from matrix perturbation. Over the past decades, many researchers have investigated the stability of an $L^{2}$-orthogonal projection and developed various upper bounds to characterize the deviation of an $L^{2}$-orthogonal projection after perturbation, which can be found, e.g., in~\cite{Stewart1977,Sun1984,Stewart1990,Sun2001,Li2013,Chen2016,Li2018}.

Let $\mathbb{C}^{m\times n}$, $\mathbb{C}_{r}^{m\times n}$, and $\mathscr{U}_{n}$ be the set of all $m\times n$ complex matrices, the set of all $m\times n$ complex matrices of rank $r$, and the set of all $n\times n$ unitary matrices, respectively. For any $M\in\mathbb{C}^{m\times n}$, the symbols $M^{\ast}$, $M^{\dagger}$, $\rank(M)$, $\|M\|_{\mathscr{U}}$, $\|M\|_{F}$, $\|M\|_{2}$, and $P_{M}$ denote the conjugate transpose, the Moore--Penrose inverse, the rank, the unitarily invariant norm (see, e.g.,~\cite[Page 357]{Horn2013}), the Frobenius norm, the spectral norm, and the $L^{2}$-orthogonal projection onto the column space of $M$ (i.e., $P_{M}=MM^{\dagger}$), respectively.

Let $A\in\mathbb{C}^{m\times n}_{r}$, $B\in\mathbb{C}^{m\times n}_{s}$, and $E=B-A$. Sun~\cite{Sun1984} established the following estimates:
\begin{subequations}
\begin{align}
&\|P_{B}-P_{A}\|_{\mathscr{U}}\leq\big(\|A^{\dagger}\|_{2}+\|B^{\dagger}\|_{2}\big)\|E\|_{\mathscr{U}},\\
&\|P_{B}-P_{A}\|_{F}^{2}\leq\big(\|A^{\dagger}\|_{2}^{2}+\|B^{\dagger}\|_{2}^{2}\big)\|E\|_{F}^{2},\label{Sun1}\\
&\|P_{B}-P_{A}\|_{2}\leq\max\big\{\|A^{\dagger}\|_{2},\|B^{\dagger}\|_{2}\big\}\|E\|_{2}.
\end{align}
\end{subequations}
In particular, if $s=r$, then
\begin{subequations}
\begin{align}
&\|P_{B}-P_{A}\|_{\mathscr{U}}\leq 2\min\big\{\|A^{\dagger}\|_{2},\|B^{\dagger}\|_{2}\big\}\|E\|_{\mathscr{U}},\\
&\|P_{B}-P_{A}\|_{F}^{2}\leq 2\min\big\{\|A^{\dagger}\|_{2}^{2},\|B^{\dagger}\|_{2}^{2}\big\}\|E\|_{F}^{2},\label{Sun2}\\
&\|P_{B}-P_{A}\|_{2}\leq\min\big\{\|A^{\dagger}\|_{2},\|B^{\dagger}\|_{2}\big\}\|E\|_{2}.
\end{align}
\end{subequations}
Recently, Chen et al.~\cite[Theorems 2.4 and 2.5]{Chen2016} improved the above estimates and proved that
\begin{subequations}
\begin{align}
&\|P_{B}-P_{A}\|_{\mathscr{U}}\leq\|EA^{\dagger}\|_{\mathscr{U}}+\|EB^{\dagger}\|_{\mathscr{U}},\\
&\|P_{B}-P_{A}\|_{F}^{2}\leq\|EA^{\dagger}\|_{F}^{2}+\|EB^{\dagger}\|_{F}^{2},\label{Chen1}\\
&\|P_{B}-P_{A}\|_{2}\leq\max\big\{\|EA^{\dagger}\|_{2}, \|EB^{\dagger}\|_{2}\big\}.
\end{align}
\end{subequations}
In particular, if $s=r$, then
\begin{subequations}
\begin{align}
&\|P_{B}-P_{A}\|_{\mathscr{U}}\leq 2\min\big\{\|EA^{\dagger}\|_{\mathscr{U}},\|EB^{\dagger}\|_{\mathscr{U}}\big\},\\
&\|P_{B}-P_{A}\|_{F}^{2}\leq 2\min\big\{\|EA^{\dagger}\|_{F}^{2},\|EB^{\dagger}\|_{F}^{2}\big\},\label{Chen2}\\
&\|P_{B}-P_{A}\|_{2}\leq\min\big\{\|EA^{\dagger}\|_{2},\|EB^{\dagger}\|_{2}\big\}.
\end{align}
\end{subequations}
Moreover, Chen et al.~\cite[Theorem 2.8]{Chen2016} derived the following combined estimate:
\begin{equation}\label{Chen-comb1}
\|P_{B}-P_{A}\|_{F}^{2}+\min\bigg\{\frac{\|A^{\dagger}\|_{2}^{2}}{\|B^{\dagger}\|_{2}^{2}},\frac{\|B^{\dagger}\|_{2}^{2}}{\|A^{\dagger}\|_{2}^{2}}\bigg\}\|P_{B^{\ast}}-P_{A^{\ast}}\|_{F}^{2}\leq\big(\|A^{\dagger}\|_{2}^{2}+\|B^{\dagger}\|_{2}^{2}\big)\|E\|_{F}^{2}.
\end{equation}
In particular, if $s=r$, then
\begin{align}
&\|P_{B}-P_{A}\|_{F}^{2}+\min\bigg\{\frac{\|A^{\dagger}\|_{2}^{2}}{\|B^{\dagger}\|_{2}^{2}},\frac{\|B^{\dagger}\|_{2}^{2}}{\|A^{\dagger}\|_{2}^{2}}\bigg\}\|P_{B^{\ast}}-P_{A^{\ast}}\|_{F}^{2}\leq 2\min\big\{\|A^{\dagger}\|_{2}^{2},\|B^{\dagger}\|_{2}^{2}\big\}\|E\|_{F}^{2},\label{Chen-comb2}\\
&\|P_{B}-P_{A}\|_{F}^{2}+\|P_{B^{\ast}}-P_{A^{\ast}}\|_{F}^{2}\leq\frac{4\|A^{\dagger}\|_{2}^{2}\|B^{\dagger}\|_{2}^{2}}{\|A^{\dagger}\|_{2}^{2}+\|B^{\dagger}\|_{2}^{2}}\|E\|_{F}^{2}.\label{Chen-comb3}
\end{align}
More recently, Li et al.~\cite[Corollary 2.4]{Li2018} showed that
\begin{equation}\label{Li1}
\|P_{B}-P_{A}\|_{F}^{2}\leq\big(\|A^{\dagger}\|_{2}^{2}+\|B^{\dagger}\|_{2}^{2}\big)\|E\|_{F}^{2}-\frac{\|B^{\dagger}\|_{2}^{2}}{\|A^{\dagger}\|_{2}^{2}}\|A^{\dagger}E\|_{F}^{2}-\frac{\|A^{\dagger}\|_{2}^{2}}{\|B^{\dagger}\|_{2}^{2}}\|B^{\dagger}E\|_{F}^{2}.
\end{equation}
In particular, if $s=r$, then
\begin{equation}\label{Li2}
\|P_{B}-P_{A}\|_{F}^{2}\leq 2\min\bigg\{\|B^{\dagger}\|_{2}^{2}\|E\|_{F}^{2}-\frac{\|B^{\dagger}\|_{2}^{2}}{\|A^{\dagger}\|_{2}^{2}}\|A^{\dagger}E\|_{F}^{2}, \|A^{\dagger}\|_{2}^{2}\|E\|_{F}^{2}-\frac{\|A^{\dagger}\|_{2}^{2}}{\|B^{\dagger}\|_{2}^{2}}\|B^{\dagger}E\|_{F}^{2}\bigg\}.
\end{equation}
In addition, Li et al.~\cite[Theorem 2.5]{Li2018} obtained the following combined estimate:
\begin{equation}\label{Li-comb1}
\|P_{B}-P_{A}\|_{F}^{2}+\|P_{B^{\ast}}-P_{A^{\ast}}\|_{F}^{2}\leq 2\max\big\{\|A^{\dagger}\|_{2}^{2},\|B^{\dagger}\|_{2}^{2}\big\}\|E\|_{F}^{2}-\frac{\|A^{\dagger}EB^{\dagger}\|_{F}^{2}+\|B^{\dagger}EA^{\dagger}\|_{F}^{2}}{\min\big\{\|A^{\dagger}\|_{2}^{2},\|B^{\dagger}\|_{2}^{2}\big\}}.
\end{equation}
In particular, if $s=r$, then
\begin{equation}\label{Li-comb2}
\|P_{B}-P_{A}\|_{F}^{2}+\|P_{B^{\ast}}-P_{A^{\ast}}\|_{F}^{2}\leq\frac{4\|A^{\dagger}\|_{2}^{2}\|B^{\dagger}\|_{2}^{2}}{\|A^{\dagger}\|_{2}^{2}+\|B^{\dagger}\|_{2}^{2}}\|E\|_{F}^{2}-\frac{2\big(\|A^{\dagger}EB^{\dagger}\|_{F}^{2}+\|B^{\dagger}EA^{\dagger}\|_{F}^{2}\big)}{\|A^{\dagger}\|_{2}^{2}+\|B^{\dagger}\|_{2}^{2}}.
\end{equation}

Although the estimate~\eqref{Chen1} has improved~\eqref{Sun1}, the upper bound in~\eqref{Chen1} is still too large in certain cases. We now give a simple example:
\begin{equation}\label{Ex0}
A=\begin{pmatrix}
1 & 0 \\
0 & 0
\end{pmatrix}, \quad B=\begin{pmatrix}
\frac{\varepsilon}{1+\varepsilon} & 0 \\
0 & \frac{\varepsilon}{10}
\end{pmatrix},
\end{equation}
where $0<\varepsilon<1$. In this example, it holds that $\|P_{B}-P_{A}\|_{F}^{2}\equiv1$. Direct computation yields that the upper bound in~\eqref{Chen1} is
\begin{displaymath}
1+\frac{1}{\varepsilon^{2}}+\frac{1}{(1+\varepsilon)^{2}},
\end{displaymath}
which is very large if $0<\varepsilon\ll1$. Alternatively, applying~\eqref{Li1} to the above example, we have that the upper bound for $\|P_{B}-P_{A}\|_{F}^{2}$ is
\begin{displaymath}
\frac{99}{100}+\frac{1}{(1+\varepsilon)^{2}}.
\end{displaymath}
Obviously, under the setting of~\eqref{Ex0}, the upper bound in~\eqref{Li1} is smaller than that in~\eqref{Chen1}. In~\cite{Li2018}, Li et al. also demonstrated the superiority of~\eqref{Li1} (compared with~\eqref{Chen1}) via some examples. However, it is difficult to compare~\eqref{Li1} with~\eqref{Chen1} theoretically. Actually, the estimate~\eqref{Li1} is not always sharper than~\eqref{Chen1}, which can be illustrated by the following example:
\begin{displaymath}
A=\begin{pmatrix}
1 & 0 \\
0 & 0
\end{pmatrix}, \quad B=\begin{pmatrix}
\frac{1}{2} & 1 \\
0 & 1
\end{pmatrix}.
\end{displaymath}
Direct calculations yield that the upper bounds in~\eqref{Chen1} and~\eqref{Li1} are $\frac{25}{4}$ and $\frac{18+3\sqrt{65}}{4}$, respectively. Therefore, there is no determined relation between the estimates~\eqref{Chen1} and~\eqref{Li1}.

Motivated by these observations, we revisit the perturbation of an $L^{2}$-orthogonal projection under the Frobenius norm. In this paper, we establish new upper bounds for $\|P_{B}-P_{A}\|_{F}^{2}$, which include the counterparts of~\eqref{Chen1}, \eqref{Chen2}, \eqref{Li1}, and~\eqref{Li2}. Some new combined upper bounds for $\|P_{B}-P_{A}\|_{F}^{2}$ and $\|P_{B^{\ast}}-P_{A^{\ast}}\|_{F}^{2}$ are also derived, which contain the counterparts of~\eqref{Chen-comb1}, \eqref{Chen-comb2}, \eqref{Chen-comb3}, \eqref{Li-comb1}, and~\eqref{Li-comb2}. Theoretical analysis shows that the new upper bounds are sharper than the existing ones. On the other hand, we also develop novel lower bounds for $\|P_{B}-P_{A}\|_{F}^{2}$  and combined lower bounds for $\|P_{B}-P_{A}\|_{F}^{2}$ and $\|P_{B^{\ast}}-P_{A^{\ast}}\|_{F}^{2}$. Furthermore, we give two examples to illustrate the performances of our theoretical results.

The rest of this paper is organized as follows. In Section~\ref{sec:pre}, we introduce a trace inequality and several identities on $\|P_{B}-P_{A}\|_{F}^{2}$ and $\|P_{B^{\ast}}-P_{A^{\ast}}\|_{F}^{2}$. In Section~\ref{sec:main}, we present some new perturbation bounds for $\|P_{B}-P_{A}\|_{F}^{2}$ and $\|P_{B^{\ast}}-P_{A^{\ast}}\|_{F}^{2}$, which involve upper bounds, lower bounds, combined upper bounds, and combined lower bounds. In Section~\ref{sec:numer}, we exhibit some numerical comparisons between the new bounds and the existing ones.

\section{Preliminaries}

\label{sec:pre}

\setcounter{equation}{0}

In this section, we introduce a useful trace inequality and several important identities on the deviations $\|P_{B}-P_{A}\|_{F}^{2}$ and $\|P_{B^{\ast}}-P_{A^{\ast}}\|_{F}^{2}$.

Let $M\in\mathbb{C}^{n\times n}$ and $N\in\mathbb{C}^{n\times n}$ be Hermitian matrices. The following lemma provides an estimate for the trace of $MN$ (see, e.g.,~\cite[Theorem 4.3.53]{Horn2013}).

\begin{lemma}
Let $\{\lambda_{i}\}_{i=1}^{n}$ and $\{\mu_{i}\}_{i=1}^{n}$ be the spectra of the Hermitian matrices $M\in\mathbb{C}^{n\times n}$ and $N\in\mathbb{C}^{n\times n}$, respectively, where $\lambda_{1}\geq\cdots\geq\lambda_{n}$ and $\mu_{1}\geq\cdots\geq\mu_{n}$. Then
\begin{equation}\label{trace}
\sum_{i=1}^{n}\lambda_{i}\mu_{n-i+1}\leq\tr(MN)\leq\sum_{i=1}^{n}\lambda_{i}\mu_{i}.
\end{equation}
\end{lemma}

Using the singular value decomposition (SVD) of a matrix, we can derive some identities on $\|P_{B}-P_{A}\|_{F}^{2}$ and $\|P_{B^{\ast}}-P_{A^{\ast}}\|_{F}^{2}$. Let $A\in\mathbb{C}^{m\times n}_{r}$ and $B\in\mathbb{C}^{m\times n}_{s}$ (\textit{throughout this paper, we only consider the nontrivial case that $r\geq 1$ and $s\geq 1$}) have the following SVDs:
\begin{subequations}
\begin{align}
&A=U\begin{pmatrix}
\Sigma_{1} & 0 \\
0 & 0
\end{pmatrix}
V^{\ast}=U_{1}\Sigma_{1}V_{1}^{\ast},\label{SVD-A}\\ &B=\widetilde{U}\begin{pmatrix}
\widetilde{\Sigma}_{1} & 0 \\
0 & 0
\end{pmatrix}
\widetilde{V}^{\ast}=\widetilde{U}_{1}\widetilde{\Sigma}_{1}\widetilde{V}_{1}^{\ast},\label{SVD-B}
\end{align}
\end{subequations}
where $U=(U_{1},U_{2})\in\mathscr{U}_{m}$, $V=(V_{1},V_{2})\in\mathscr{U}_{n}$, $\widetilde{U}=(\widetilde{U}_{1},\widetilde{U}_{2})\in\mathscr{U}_{m}$, $\widetilde{V}=(\widetilde{V}_{1},\widetilde{V}_{2})\in\mathscr{U}_{n}$, $U_{1}\in\mathbb{C}^{m\times r}$, $V_{1}\in\mathbb{C}^{n\times r}$, $\widetilde{U}_{1}\in\mathbb{C}^{m\times s}$, $\widetilde{V}_{1}\in\mathbb{C}^{n\times s}$, $\Sigma_{1}=\diag(\sigma_{1},\ldots,\sigma_{r})$, $\widetilde{\Sigma}_{1}=\diag(\widetilde{\sigma}_{1},\ldots,\widetilde{\sigma}_{s})$, $\sigma_{1}\geq\cdots\geq\sigma_{r}>0$, and $\widetilde{\sigma}_{1}\geq\cdots\geq\widetilde{\sigma}_{s}>0$. In view of~\eqref{SVD-A} and~\eqref{SVD-B}, the Moore--Penrose inverses $A^{\dagger}$ and $B^{\dagger}$ can be explicitly expressed as follows:
\begin{subequations}
\begin{align}
&A^{\dagger}=V\begin{pmatrix}
\Sigma_{1}^{-1} & 0 \\
0 & 0
\end{pmatrix}
U^{\ast}=V_{1}\Sigma_{1}^{-1}U_{1}^{\ast},\label{MP-A}\\ &B^{\dagger}=\widetilde{V}\begin{pmatrix}
\widetilde{\Sigma}_{1}^{-1} & 0 \\
0 & 0
\end{pmatrix}
\widetilde{U}^{\ast}=\widetilde{V}_{1}\widetilde{\Sigma}_{1}^{-1}\widetilde{U}_{1}^{\ast}.\label{MP-B}
\end{align}
\end{subequations}
By~\eqref{SVD-A}, \eqref{SVD-B}, \eqref{MP-A}, and~\eqref{MP-B}, we have
\begin{displaymath}
P_{A}=AA^{\dagger}=U_{1}U_{1}^{\ast},\quad P_{A^{\ast}}=A^{\dagger}A=V_{1}V_{1}^{\ast},\quad P_{B}=BB^{\dagger}=\widetilde{U}_{1}\widetilde{U}_{1}^{\ast},\quad P_{B^{\ast}}=B^{\dagger}B=\widetilde{V}_{1}\widetilde{V}_{1}^{\ast}.
\end{displaymath}

The following lemma (see~\cite[Lemma 2.3]{Chen2016}) is the foundation of our analysis, which gives the expressions for $\|P_{B}-P_{A}\|_{F}^{2}$ and $\|P_{B^{\ast}}-P_{A^{\ast}}\|_{F}^{2}$.

\begin{lemma}\label{expression}
Let $A\in\mathbb{C}^{m\times n}_{r}$ and $B\in\mathbb{C}^{m\times n}_{s}$ have the SVDs~\eqref{SVD-A} and~\eqref{SVD-B}, respectively. Then
\begin{subequations}
\begin{align}
&\|P_{B}-P_{A}\|_{F}^{2}=\|\widetilde{U}_{1}^{\ast}U_{2}\|_{F}^{2}+\|\widetilde{U}_{2}^{\ast}U_{1}\|_{F}^{2},\label{exp1.1}\\
&\|P_{B^{\ast}}-P_{A^{\ast}}\|_{F}^{2}=\|\widetilde{V}_{1}^{\ast}V_{2}\|_{F}^{2}+\|\widetilde{V}_{2}^{\ast}V_{1}\|_{F}^{2}.\label{exp1.2}
\end{align}
\end{subequations}
In particular, if $s=r$, then
\begin{subequations}
\begin{align}
&\|P_{B}-P_{A}\|_{F}^{2}=2\|\widetilde{U}_{1}^{\ast}U_{2}\|_{F}^{2}=2\|\widetilde{U}_{2}^{\ast}U_{1}\|_{F}^{2},\label{exp2.1}\\
&\|P_{B^{\ast}}-P_{A^{\ast}}\|_{F}^{2}=2\|\widetilde{V}_{1}^{\ast}V_{2}\|_{F}^{2}=2\|\widetilde{V}_{2}^{\ast}V_{1}\|_{F}^{2}.\label{exp2.2}
\end{align}
\end{subequations}
\end{lemma}

Based on Lemma~\ref{expression}, we can get the following identities on $\|P_{B}-P_{A}\|_{F}^{2}$ and $\|P_{B^{\ast}}-P_{A^{\ast}}\|_{F}^{2}$, which do not involve the auxiliary matrices $U_{i}$, $\widetilde{U}_{i}$, $V_{i}$, and $\widetilde{V}_{i}$ ($i=1,2$).

\begin{lemma}\label{identity}
Let $A\in\mathbb{C}^{m\times n}_{r}$, $B\in\mathbb{C}^{m\times n}_{s}$, and $E=B-A$. Then
\begin{subequations}
\begin{align}
&\|P_{B}-P_{A}\|_{F}^{2}=\|EA^{\dagger}\|_{F}^{2}+\|EB^{\dagger}\|_{F}^{2}-\|BB^{\dagger}EA^{\dagger}\|_{F}^{2}-\|AA^{\dagger}EB^{\dagger}\|_{F}^{2},\label{ide1.1}\\
&\|P_{B^{\ast}}-P_{A^{\ast}}\|_{F}^{2}=\|A^{\dagger}E\|_{F}^{2}+\|B^{\dagger}E\|_{F}^{2}-\|A^{\dagger}EB^{\dagger}B\|_{F}^{2}-\|B^{\dagger}EA^{\dagger}A\|_{F}^{2}.\label{ide1.2}
\end{align}
\end{subequations}
In particular, if $s=r$, then
\begin{subequations}
\begin{align}
&\|P_{B}-P_{A}\|_{F}^{2}=2\big(\|EA^{\dagger}\|_{F}^{2}-\|BB^{\dagger}EA^{\dagger}\|_{F}^{2}\big)=2\big(\|EB^{\dagger}\|_{F}^{2}-\|AA^{\dagger}EB^{\dagger}\|_{F}^{2}\big),\label{ide2.1}\\
&\|P_{B^{\ast}}-P_{A^{\ast}}\|_{F}^{2}=2\big(\|A^{\dagger}E\|_{F}^{2}-\|A^{\dagger}EB^{\dagger}B\|_{F}^{2}\big)=2\big(\|B^{\dagger}E\|_{F}^{2}-\|B^{\dagger}EA^{\dagger}A\|_{F}^{2}\big).\label{ide2.2}
\end{align}
\end{subequations}
\end{lemma}

\begin{proof}
By~\eqref{SVD-A}, \eqref{SVD-B}, \eqref{MP-A}, and~\eqref{MP-B}, we have
\begin{align*}
&\widetilde{U}^{\ast}EA^{\dagger}U=\begin{pmatrix}
\widetilde{\Sigma}_{1}\widetilde{V}_{1}^{\ast}V_{1}\Sigma_{1}^{-1}-\widetilde{U}_{1}^{\ast}U_{1} & 0\\
-\widetilde{U}_{2}^{\ast}U_{1} & 0
\end{pmatrix},\\
&\widetilde{U}^{\ast}BB^{\dagger}EA^{\dagger}U=\begin{pmatrix}
\widetilde{\Sigma}_{1}\widetilde{V}_{1}^{\ast}V_{1}\Sigma_{1}^{-1}-\widetilde{U}_{1}^{\ast}U_{1} & 0\\
0 & 0
\end{pmatrix}.
\end{align*}
Hence,
\begin{align}
&\|EA^{\dagger}\|_{F}^{2}=\|\widetilde{\Sigma}_{1}\widetilde{V}_{1}^{\ast}V_{1}\Sigma_{1}^{-1}-\widetilde{U}_{1}^{\ast}U_{1}\|_{F}^{2}+\|\widetilde{U}_{2}^{\ast}U_{1}\|_{F}^{2},\label{rela1.1}\\
&\|BB^{\dagger}EA^{\dagger}\|_{F}^{2}=\|\widetilde{\Sigma}_{1}\widetilde{V}_{1}^{\ast}V_{1}\Sigma_{1}^{-1}-\widetilde{U}_{1}^{\ast}U_{1}\|_{F}^{2}.\label{rela1.2}
\end{align}
Using~\eqref{rela1.1} and~\eqref{rela1.2}, we obtain
\begin{equation}\label{part1}
\|\widetilde{U}_{2}^{\ast}U_{1}\|_{F}^{2}=\|EA^{\dagger}\|_{F}^{2}-\|BB^{\dagger}EA^{\dagger}\|_{F}^{2}.
\end{equation}
Similarly, we have
\begin{align*}
&U^{\ast}EB^{\dagger}\widetilde{U}=\begin{pmatrix}
U_{1}^{\ast}\widetilde{U}_{1}-\Sigma_{1}V_{1}^{\ast}\widetilde{V}_{1}\widetilde{\Sigma}_{1}^{-1} & 0\\
U_{2}^{\ast}\widetilde{U}_{1} & 0
\end{pmatrix},\\
&U^{\ast}AA^{\dagger}EB^{\dagger}\widetilde{U}=\begin{pmatrix}
U_{1}^{\ast}\widetilde{U}_{1}-\Sigma_{1}V_{1}^{\ast}\widetilde{V}_{1}\widetilde{\Sigma}_{1}^{-1} & 0\\
0 & 0
\end{pmatrix}.
\end{align*}
Thus,
\begin{align}
&\|EB^{\dagger}\|_{F}^{2}=\|U_{1}^{\ast}\widetilde{U}_{1}-\Sigma_{1}V_{1}^{\ast}\widetilde{V}_{1}\widetilde{\Sigma}_{1}^{-1}\|_{F}^{2}+\|\widetilde{U}_{1}^{\ast}U_{2}\|_{F}^{2},\label{rela2.1}\\
&\|AA^{\dagger}EB^{\dagger}\|_{F}^{2}=\|U_{1}^{\ast}\widetilde{U}_{1}-\Sigma_{1}V_{1}^{\ast}\widetilde{V}_{1}\widetilde{\Sigma}_{1}^{-1}\|_{F}^{2}.\label{rela2.2}
\end{align}
From~\eqref{rela2.1} and~\eqref{rela2.2}, we have
\begin{equation}\label{part2}
\|\widetilde{U}_{1}^{\ast}U_{2}\|_{F}^{2}=\|EB^{\dagger}\|_{F}^{2}-\|AA^{\dagger}EB^{\dagger}\|_{F}^{2}.
\end{equation}
The identity~\eqref{ide1.1} then follows by combining~\eqref{exp1.1}, \eqref{part1}, and~\eqref{part2}. In particular, if $s=r$, using~\eqref{exp2.1}, \eqref{part1}, and~\eqref{part2}, we can obtain the identity~\eqref{ide2.1}.

Replacing $A$ and $B$ in~\eqref{ide1.1} by $A^{\ast}$ and $B^{\ast}$, respectively, we can arrive at the identity~\eqref{ide1.2}. Analogously, the identity~\eqref{ide2.2} can be deduced from~\eqref{ide2.1}. This completes the proof.
\end{proof}

On the basis of Lemma~\ref{identity}, we can easily get the following corollary.

\begin{corollary}\label{identity-c}
Let $A\in\mathbb{C}^{m\times n}_{r}$, $B\in\mathbb{C}^{m\times n}_{s}$, and $\widetilde{E}=B^{\dagger}-A^{\dagger}$. Then
\begin{align*}
&\|P_{B}-P_{A}\|_{F}^{2}=\|A\widetilde{E}\|_{F}^{2}+\|B\widetilde{E}\|_{F}^{2}-\|A\widetilde{E}BB^{\dagger}\|_{F}^{2}-\|B\widetilde{E}AA^{\dagger}\|_{F}^{2},\\
&\|P_{B^{\ast}}-P_{A^{\ast}}\|_{F}^{2}=\|\widetilde{E}A\|_{F}^{2}+\|\widetilde{E}B\|_{F}^{2}-\|B^{\dagger}B\widetilde{E}A\|_{F}^{2}-\|A^{\dagger}A\widetilde{E}B\|_{F}^{2}.
\end{align*}
In particular, if $s=r$, then
\begin{align*}
&\|P_{B}-P_{A}\|_{F}^{2}=2\big(\|A\widetilde{E}\|_{F}^{2}-\|A\widetilde{E}BB^{\dagger}\|_{F}^{2}\big)=2\big(\|B\widetilde{E}\|_{F}^{2}-\|B\widetilde{E}AA^{\dagger}\|_{F}^{2}\big),\\
&\|P_{B^{\ast}}-P_{A^{\ast}}\|_{F}^{2}=2\big(\|\widetilde{E}A\|_{F}^{2}-\|B^{\dagger}B\widetilde{E}A\|_{F}^{2}\big)=2\big(\|\widetilde{E}B\|_{F}^{2}-\|A^{\dagger}A\widetilde{E}B\|_{F}^{2}\big).
\end{align*}
\end{corollary}

In what follows, we will apply Lemmas~\ref{expression} and~\ref{identity} to establish the perturbation bounds for an $L^{2}$-orthogonal projection. The corresponding results based on Corollary~\ref{identity-c} can be derived in a similar manner.

\section{Main results}

\label{sec:main}

\setcounter{equation}{0}

In this section, we present new upper and lower bounds for $\|P_{B}-P_{A}\|_{F}^{2}$. Some novel combined upper and lower bounds for $\|P_{B}-P_{A}\|_{F}^{2}$ and $\|P_{B^{\ast}}-P_{A^{\ast}}\|_{F}^{2}$ are also developed. We mention that the upper and lower bounds for $\|P_{B^{\ast}}-P_{A^{\ast}}\|_{F}^{2}$ will be omitted, because they can be directly deduced from that for $\|P_{B}-P_{A}\|_{F}^{2}$.

We first give an estimate for $\|P_{B}-P_{A}\|_{F}^{2}$, which depends only on the ranks of $A$ and $B$.

\begin{theorem}
Let $A\in\mathbb{C}^{m\times n}_{r}$ and $B\in\mathbb{C}^{m\times n}_{s}$.

{\rm (i)} If $s+r\leq m$, then
\begin{equation}\label{rank1.1}
|s-r|\leq\|P_{B}-P_{A}\|_{F}^{2}\leq s+r.
\end{equation}

{\rm (ii)} If $s+r>m$, then
\begin{equation}\label{rank1.2}
|s-r|\leq\|P_{B}-P_{A}\|_{F}^{2}\leq 2m-s-r.
\end{equation}
\end{theorem}

\begin{proof}
Since both $P_{A}$ and $P_{B}$ are Hermitian and idempotent, we have
\begin{displaymath}
\|P_{B}-P_{A}\|_{F}^{2}=\tr(P_{B}+P_{A}-P_{B}P_{A}-P_{A}P_{B})=s+r-2\tr(P_{B}P_{A}),
\end{displaymath}
where we have used the fact that the trace of an idempotent matrix equals its rank.

If $s+r\leq m$, by~\eqref{trace}, we have
\begin{displaymath}
0\leq\tr(P_{B}P_{A})\leq\min\{s,r\},
\end{displaymath}
which yields
\begin{displaymath}
|s-r|\leq\|P_{B}-P_{A}\|_{F}^{2}\leq s+r.
\end{displaymath}
On the other hand, if $s+r>m$, then
\begin{displaymath}
s+r-m\leq\tr(P_{B}P_{A})\leq\min\{s,r\},
\end{displaymath}
which leads to
\begin{displaymath}
|s-r|\leq\|P_{B}-P_{A}\|_{F}^{2}\leq 2m-s-r.
\end{displaymath}
This completes the proof.
\end{proof}

\begin{remark}\rm
According to the lower bounds in~\eqref{rank1.1} and~\eqref{rank1.2}, we deduce that a necessary condition for $\lim\limits_{B\rightarrow A}P_{B}=P_{A}$ ($B$ is viewed as a variable) is that $\rank(B)=\rank(A)$ always holds when $B$ tends to $A$. Indeed, it is also a sufficient condition for $\lim\limits_{B\rightarrow A}P_{B}=P_{A}$ (see~\cite{Sun1984,Sun2001}).
\end{remark}

In what follows, we develop some perturbation bounds involving the matrices $E=B-A$ and $\widetilde{E}=B^{\dagger}-A^{\dagger}$.

\subsection{Upper bounds}

\label{subsec:up}

In this subsection, we present several new upper bounds for $\|P_{B}-P_{A}\|_{F}^{2}$, which improve the existing results.

On the basis of~\eqref{ide1.1} and~\eqref{ide2.1}, we can derive the following estimates for $\|P_{B}-P_{A}\|_{F}^{2}$, which are sharper than~\eqref{Chen1} and~\eqref{Chen2}.

\begin{theorem}\label{up-thm1}
Let $A\in\mathbb{C}^{m\times n}_{r}$, $B\in\mathbb{C}^{m\times n}_{s}$, $E=B-A$, and $\widetilde{E}=B^{\dagger}-A^{\dagger}$. Define
\begin{align*}
&\alpha_{1}:=\max\bigg\{\frac{\|B^{\dagger}EA^{\dagger}\|_{F}^{2}}{\|B^{\dagger}\|_{2}^{2}},\frac{\|B\widetilde{E}A\|_{F}^{2}}{\|A\|_{2}^{2}}\bigg\},\\ &\alpha_{2}:=\max\bigg\{\frac{\|A^{\dagger}EB^{\dagger}\|_{F}^{2}}{\|A^{\dagger}\|_{2}^{2}},\frac{\|A\widetilde{E}B\|_{F}^{2}}{\|B\|_{2}^{2}}\bigg\}.
\end{align*}
Then
\begin{equation}\label{up1.1}
\|P_{B}-P_{A}\|_{F}^{2}\leq\|EA^{\dagger}\|_{F}^{2}+\|EB^{\dagger}\|_{F}^{2}-\alpha_{1}-\alpha_{2}.
\end{equation}
In particular, if $s=r$, then
\begin{equation}\label{up1.2}
\|P_{B}-P_{A}\|_{F}^{2}\leq 2\min\big\{\|EA^{\dagger}\|_{F}^{2}-\alpha_{1},\|EB^{\dagger}\|_{F}^{2}-\alpha_{2}\big\}.
\end{equation}
\end{theorem}

\begin{proof}
Using~\eqref{SVD-A}, \eqref{SVD-B}, \eqref{MP-A}, and~\eqref{MP-B}, we obtain
\begin{align*}
&\widetilde{V}^{\ast}B^{\dagger}EA^{\dagger}U=\begin{pmatrix}
\widetilde{V}_{1}^{\ast}V_{1}\Sigma_{1}^{-1}-\widetilde{\Sigma}_{1}^{-1}\widetilde{U}_{1}^{\ast}U_{1} & 0 \\
0 & 0
\end{pmatrix},\\
&V^{\ast}A^{\dagger}EB^{\dagger}\widetilde{U}=\begin{pmatrix}
\Sigma_{1}^{-1}U_{1}^{\ast}\widetilde{U}_{1}-V_{1}^{\ast}\widetilde{V}_{1}\widetilde{\Sigma}_{1}^{-1} & 0 \\
0 & 0
\end{pmatrix}.
\end{align*}
Thus,
\begin{align}
&\|B^{\dagger}EA^{\dagger}\|_{F}^{2}=\|\widetilde{V}_{1}^{\ast}V_{1}\Sigma_{1}^{-1}-\widetilde{\Sigma}_{1}^{-1}\widetilde{U}_{1}^{\ast}U_{1}\|_{F}^{2},\label{B+EA+}\\
&\|A^{\dagger}EB^{\dagger}\|_{F}^{2}=\|\Sigma_{1}^{-1}U_{1}^{\ast}\widetilde{U}_{1}-V_{1}^{\ast}\widetilde{V}_{1}\widetilde{\Sigma}_{1}^{-1}\|_{F}^{2}.\label{A+EB+}
\end{align}
According to~\eqref{rela1.2}, \eqref{rela2.2}, \eqref{B+EA+}, and~\eqref{A+EB+}, we deduce that
\begin{align}
&\|BB^{\dagger}EA^{\dagger}\|_{F}^{2}=\|\widetilde{\Sigma}_{1}(\widetilde{V}_{1}^{\ast}V_{1}\Sigma_{1}^{-1}-\widetilde{\Sigma}_{1}^{-1}\widetilde{U}_{1}^{\ast}U_{1})\|_{F}^{2}\geq\frac{\|B^{\dagger}EA^{\dagger}\|_{F}^{2}}{\|B^{\dagger}\|_{2}^{2}},\label{BB+EA+1}\\
&\|AA^{\dagger}EB^{\dagger}\|_{F}^{2}=\|\Sigma_{1}(\Sigma_{1}^{-1}U_{1}^{\ast}\widetilde{U}_{1}-V_{1}^{\ast}\widetilde{V}_{1}\widetilde{\Sigma}_{1}^{-1})\|_{F}^{2}\geq\frac{\|A^{\dagger}EB^{\dagger}\|_{F}^{2}}{\|A^{\dagger}\|_{2}^{2}}.\label{AA+EB+1}
\end{align}
Similarly, we have
\begin{align*}
&\widetilde{U}^{\ast}B\widetilde{E}AV=\begin{pmatrix}
\widetilde{U}_{1}^{\ast}U_{1}\Sigma_{1}-\widetilde{\Sigma}_{1}\widetilde{V}_{1}^{\ast}V_{1} & 0 \\
0 & 0
\end{pmatrix},\\
&U^{\ast}A\widetilde{E}B\widetilde{V}=\begin{pmatrix}
\Sigma_{1}V_{1}^{\ast}\widetilde{V}_{1}-U_{1}^{\ast}\widetilde{U}_{1}\widetilde{\Sigma}_{1} & 0 \\
0 & 0
\end{pmatrix}.
\end{align*}
Hence,
\begin{align}
&\|B\widetilde{E}A\|_{F}^{2}=\|\widetilde{\Sigma}_{1}\widetilde{V}_{1}^{\ast}V_{1}-\widetilde{U}_{1}^{\ast}U_{1}\Sigma_{1}\|_{F}^{2},\label{BE+A}\\
&\|A\widetilde{E}B\|_{F}^{2}=\|U_{1}^{\ast}\widetilde{U}_{1}\widetilde{\Sigma}_{1}-\Sigma_{1}V_{1}^{\ast}\widetilde{V}_{1}\|_{F}^{2}.\label{AE+B}
\end{align}
From~\eqref{rela1.2}, \eqref{rela2.2}, \eqref{BE+A}, and~\eqref{AE+B}, we deduce that
\begin{align}
&\|BB^{\dagger}EA^{\dagger}\|_{F}^{2}=\|(\widetilde{\Sigma}_{1}\widetilde{V}_{1}^{\ast}V_{1}-\widetilde{U}_{1}^{\ast}U_{1}\Sigma_{1})\Sigma_{1}^{-1}\|_{F}^{2}\geq\frac{\|B\widetilde{E}A\|_{F}^{2}}{\|A\|_{2}^{2}},\label{BB+EA+2}\\
&\|AA^{\dagger}EB^{\dagger}\|_{F}^{2}=\|(U_{1}^{\ast}\widetilde{U}_{1}\widetilde{\Sigma}_{1}-\Sigma_{1}V_{1}^{\ast}\widetilde{V}_{1})\widetilde{\Sigma}_{1}^{-1}\|_{F}^{2}\geq\frac{\|A\widetilde{E}B\|_{F}^{2}}{\|B\|_{2}^{2}}.\label{AA+EB+2}
\end{align}
Based on~\eqref{BB+EA+1}, \eqref{AA+EB+1}, \eqref{BB+EA+2}, and~\eqref{AA+EB+2}, we arrive at
\begin{align}
&\|BB^{\dagger}EA^{\dagger}\|_{F}^{2}\geq\max\bigg\{\frac{\|B^{\dagger}EA^{\dagger}\|_{F}^{2}}{\|B^{\dagger}\|_{2}^{2}},\frac{\|B\widetilde{E}A\|_{F}^{2}}{\|A\|_{2}^{2}}\bigg\},\label{BB+EA+}\\
&\|AA^{\dagger}EB^{\dagger}\|_{F}^{2}\geq\max\bigg\{\frac{\|A^{\dagger}EB^{\dagger}\|_{F}^{2}}{\|A^{\dagger}\|_{2}^{2}},\frac{\|A\widetilde{E}B\|_{F}^{2}}{\|B\|_{2}^{2}}\bigg\}.\label{AA+EB+}
\end{align}

The inequality~\eqref{up1.1} then follows by combining~\eqref{ide1.1}, \eqref{BB+EA+}, and~\eqref{AA+EB+}. In particular, if $s=r$, using~\eqref{ide2.1}, \eqref{BB+EA+}, and~\eqref{AA+EB+}, we can obtain the inequality~\eqref{up1.2}.
\end{proof}

Based on~\eqref{exp1.1} and~\eqref{exp2.1}, we can derive the following theorem.

\begin{theorem}\label{up-thm2}
Let $A\in\mathbb{C}^{m\times n}_{r}$, $B\in\mathbb{C}^{m\times n}_{s}$, $E=B-A$, and $\widetilde{E}=B^{\dagger}-A^{\dagger}$. Define
\begin{align*}
&\beta_{1}:=\min\Big\{\|A^{\dagger}\|_{2}^{2}\big(\|E\|_{F}^{2}-\|BB^{\dagger}E\|_{F}^{2}\big),\|A\|_{2}^{2}\big(\|\widetilde{E}\|_{F}^{2}-\|\widetilde{E}BB^{\dagger}\|_{F}^{2}\big)\Big\},\\
&\beta_{2}:=\min\Big\{\|B^{\dagger}\|_{2}^{2}\big(\|E\|_{F}^{2}-\|AA^{\dagger}E\|_{F}^{2}\big),\|B\|_{2}^{2}\big(\|\widetilde{E}\|_{F}^{2}-\|\widetilde{E}AA^{\dagger}\|_{F}^{2}\big)\Big\}.
\end{align*}
Then
\begin{equation}\label{up2.1}
\|P_{B}-P_{A}\|_{F}^{2}\leq\beta_{1}+\beta_{2}.
\end{equation}
In particular, if $s=r$, then
\begin{equation}\label{up2.2}
\|P_{B}-P_{A}\|_{F}^{2}\leq 2\min\big\{\beta_{1},\beta_{2}\big\}.
\end{equation}
\end{theorem}

\begin{proof}
By~\eqref{SVD-A}, \eqref{SVD-B}, \eqref{MP-A}, and~\eqref{MP-B}, we have
\begin{align}
&U^{\ast}E\widetilde{V}=\begin{pmatrix}
U_{1}^{\ast}\widetilde{U}_{1}\widetilde{\Sigma}_{1}-\Sigma_{1}V_{1}^{\ast}\widetilde{V}_{1} & -\Sigma_{1}V_{1}^{\ast}\widetilde{V}_{2}\\
U_{2}^{\ast}\widetilde{U}_{1}\widetilde{\Sigma}_{1} & 0
\end{pmatrix},\label{E1}\\
&U^{\ast}AA^{\dagger}E\widetilde{V}=\begin{pmatrix}
U_{1}^{\ast}\widetilde{U}_{1}\widetilde{\Sigma}_{1}-\Sigma_{1}V_{1}^{\ast}\widetilde{V}_{1} & -\Sigma_{1}V_{1}^{\ast}\widetilde{V}_{2}\\
0 & 0
\end{pmatrix}.\label{AA+E}
\end{align}
From~\eqref{E1} and~\eqref{AA+E}, we deduce that
\begin{displaymath}
\|\widetilde{\Sigma}_{1}\widetilde{U}_{1}^{\ast}U_{2}\|_{F}^{2}=\|E\|_{F}^{2}-\|AA^{\dagger}E\|_{F}^{2}.
\end{displaymath}
Due to
\begin{displaymath}
\|\widetilde{U}_{1}^{\ast}U_{2}\|_{F}^{2}\leq\|B^{\dagger}\|_{2}^{2}\|\widetilde{\Sigma}_{1}\widetilde{U}_{1}^{\ast}U_{2}\|_{F}^{2},
\end{displaymath}
it follows that
\begin{displaymath}
\|\widetilde{U}_{1}^{\ast}U_{2}\|_{F}^{2}\leq\|B^{\dagger}\|_{2}^{2}\big(\|E\|_{F}^{2}-\|AA^{\dagger}E\|_{F}^{2}\big).
\end{displaymath}
In addition, we have
\begin{align}
&\widetilde{V}^{\ast}\widetilde{E}U=\begin{pmatrix}
\widetilde{\Sigma}_{1}^{-1}\widetilde{U}_{1}^{\ast}U_{1}-\widetilde{V}_{1}^{\ast}V_{1}\Sigma_{1}^{-1} & \widetilde{\Sigma}_{1}^{-1}\widetilde{U}_{1}^{\ast}U_{2} \\
-\widetilde{V}_{2}^{\ast}V_{1}\Sigma_{1}^{-1} & 0
\end{pmatrix},\label{E+1}\\
&\widetilde{V}^{\ast}\widetilde{E}AA^{\dagger}U=\begin{pmatrix}
\widetilde{\Sigma}_{1}^{-1}\widetilde{U}_{1}^{\ast}U_{1}-\widetilde{V}_{1}^{\ast}V_{1}\Sigma_{1}^{-1} & 0 \\
-\widetilde{V}_{2}^{\ast}V_{1}\Sigma_{1}^{-1} & 0
\end{pmatrix}.\label{E+AA+}
\end{align}
By~\eqref{E+1} and~\eqref{E+AA+}, we have
\begin{displaymath}
\|\widetilde{\Sigma}_{1}^{-1}\widetilde{U}_{1}^{\ast}U_{2}\|_{F}^{2}=\|\widetilde{E}\|_{F}^{2}-\|\widetilde{E}AA^{\dagger}\|_{F}^{2}.
\end{displaymath}
Since
\begin{displaymath}
\|\widetilde{U}_{1}^{\ast}U_{2}\|_{F}^{2}\leq\|B\|_{2}^{2}\|\widetilde{\Sigma}_{1}^{-1}\widetilde{U}_{1}^{\ast}U_{2}\|_{F}^{2},
\end{displaymath}
it follows that
\begin{displaymath}
\|\widetilde{U}_{1}^{\ast}U_{2}\|_{F}^{2}\leq\|B\|_{2}^{2}\big(\|\widetilde{E}\|_{F}^{2}-\|\widetilde{E}AA^{\dagger}\|_{F}^{2}\big).
\end{displaymath}
Thus,
\begin{equation}\label{U1starU2-1}
\|\widetilde{U}_{1}^{\ast}U_{2}\|_{F}^{2}\leq\min\Big\{\|B^{\dagger}\|_{2}^{2}\big(\|E\|_{F}^{2}-\|AA^{\dagger}E\|_{F}^{2}\big),\|B\|_{2}^{2}\big(\|\widetilde{E}\|_{F}^{2}-\|\widetilde{E}AA^{\dagger}\|_{F}^{2}\big)\Big\}.
\end{equation}

Similarly,
\begin{align}
&\widetilde{U}^{\ast}EV=\begin{pmatrix}
\widetilde{\Sigma}_{1}\widetilde{V}_{1}^{\ast}V_{1}-\widetilde{U}_{1}^{\ast}U_{1}\Sigma_{1} & \widetilde{\Sigma}_{1}\widetilde{V}_{1}^{\ast}V_{2}\\
-\widetilde{U}_{2}^{\ast}U_{1}\Sigma_{1} & 0
\end{pmatrix},\label{E2}\\
&\widetilde{U}^{\ast}BB^{\dagger}EV=\begin{pmatrix}
\widetilde{\Sigma}_{1}\widetilde{V}_{1}^{\ast}V_{1}-\widetilde{U}_{1}^{\ast}U_{1}\Sigma_{1} & \widetilde{\Sigma}_{1}\widetilde{V}_{1}^{\ast}V_{2}\\
0 & 0
\end{pmatrix},\label{BB+E}\\
&V^{\ast}\widetilde{E}\widetilde{U}=\begin{pmatrix}
V_{1}^{\ast}\widetilde{V}_{1}\widetilde{\Sigma}_{1}^{-1}-\Sigma_{1}^{-1}U_{1}^{\ast}\widetilde{U}_{1} & -\Sigma_{1}^{-1}U_{1}^{\ast}\widetilde{U}_{2} \\
V_{2}^{\ast}\widetilde{V}_{1}\widetilde{\Sigma}_{1}^{-1} & 0
\end{pmatrix},\label{E+2}\\
&V^{\ast}\widetilde{E}BB^{\dagger}\widetilde{U}=\begin{pmatrix}
V_{1}^{\ast}\widetilde{V}_{1}\widetilde{\Sigma}_{1}^{-1}-\Sigma_{1}^{-1}U_{1}^{\ast}\widetilde{U}_{1} & 0 \\
V_{2}^{\ast}\widetilde{V}_{1}\widetilde{\Sigma}_{1}^{-1} & 0
\end{pmatrix}.\label{E+BB+}
\end{align}
Using~\eqref{E2} and~\eqref{BB+E}, we obtain
\begin{displaymath}
\|\widetilde{U}_{2}^{\ast}U_{1}\|_{F}^{2}\leq\|A^{\dagger}\|_{2}^{2}\|\widetilde{U}_{2}^{\ast}U_{1}\Sigma_{1}\|_{F}^{2}=\|A^{\dagger}\|_{2}^{2}\big(\|E\|_{F}^{2}-\|BB^{\dagger}E\|_{F}^{2}\big).
\end{displaymath}
In light of~\eqref{E+2} and~\eqref{E+BB+}, we have
\begin{displaymath}
\|\widetilde{U}_{2}^{\ast}U_{1}\|_{F}^{2}\leq\|A\|_{2}^{2}\|\widetilde{U}_{2}^{\ast}U_{1}\Sigma_{1}^{-1}\|_{F}^{2}=\|A\|_{2}^{2}\big(\|\widetilde{E}\|_{F}^{2}-\|\widetilde{E}BB^{\dagger}\|_{F}^{2}\big).
\end{displaymath}
Hence,
\begin{equation}\label{U2starU1-1}
\|\widetilde{U}_{2}^{\ast}U_{1}\|_{F}^{2}\leq\min\Big\{\|A^{\dagger}\|_{2}^{2}\big(\|E\|_{F}^{2}-\|BB^{\dagger}E\|_{F}^{2}\big),\|A\|_{2}^{2}\big(\|\widetilde{E}\|_{F}^{2}-\|\widetilde{E}BB^{\dagger}\|_{F}^{2}\big)\Big\}.
\end{equation}

In view of~\eqref{exp1.1}, \eqref{U1starU2-1}, and~\eqref{U2starU1-1}, we conclude that the inequality~\eqref{up2.1} holds. In particular, if $s=r$, using~\eqref{exp2.1}, \eqref{U1starU2-1}, and~\eqref{U2starU1-1}, we can get the inequality~\eqref{up2.2}.
\end{proof}

\begin{remark}\rm
By~\eqref{AA+E}, we have
\begin{displaymath}
\|AA^{\dagger}E\|_{F}^{2}=\|\Sigma_{1}(\Sigma_{1}^{-1}U_{1}^{\ast}\widetilde{U}_{1}\widetilde{\Sigma}_{1}-V_{1}^{\ast}\widetilde{V}_{1})\|_{F}^{2}+\|\Sigma_{1}V_{1}^{\ast}\widetilde{V}_{2}\|_{F}^{2}\geq\frac{\|A^{\dagger}E\|_{F}^{2}}{\|A^{\dagger}\|_{2}^{2}},
\end{displaymath}
where we have used the fact that
\begin{displaymath}
\|A^{\dagger}E\|_{F}^{2}=\|\Sigma_{1}^{-1}U_{1}^{\ast}\widetilde{U}_{1}\widetilde{\Sigma}_{1}-V_{1}^{\ast}\widetilde{V}_{1}\|_{F}^{2}+\|V_{1}^{\ast}\widetilde{V}_{2}\|_{F}^{2}.
\end{displaymath}
Analogously, it holds that
\begin{displaymath}
\|BB^{\dagger}E\|_{F}^{2}\geq\frac{\|B^{\dagger}E\|_{F}^{2}}{\|B^{\dagger}\|_{2}^{2}}.
\end{displaymath}
Then
\begin{align*}
\beta_{1}&\leq\|A^{\dagger}\|_{2}^{2}\bigg(\|E\|_{F}^{2}-\frac{\|B^{\dagger}E\|_{F}^{2}}{\|B^{\dagger}\|_{2}^{2}}\bigg),\\ \beta_{2}&\leq\|B^{\dagger}\|_{2}^{2}\bigg(\|E\|_{F}^{2}-\frac{\|A^{\dagger}E\|_{F}^{2}}{\|A^{\dagger}\|_{2}^{2}}\bigg).
\end{align*}
Therefore, the estimates~\eqref{up2.1} and~\eqref{up2.2} are sharper than~\eqref{Li1} and~\eqref{Li2}, respectively.
\end{remark}

The following corollary provides an alternative version of Theorem~\ref{up-thm2}.

\begin{corollary}\label{up-cor}
Let $A\in\mathbb{C}^{m\times n}_{r}$, $B\in\mathbb{C}^{m\times n}_{s}$, $E=B-A$, and $\widetilde{E}=B^{\dagger}-A^{\dagger}$. Define
\begin{align*}
&\gamma_{1}:=\min\Big\{\|A^{\dagger}\|_{2}^{2}\big(\|EA^{\dagger}A\|_{F}^{2}-\|B\widetilde{E}A\|_{F}^{2}\big),\|A\|_{2}^{2}\big(\|A^{\dagger}A\widetilde{E}\|_{F}^{2}-\|A^{\dagger}EB^{\dagger}\|_{F}^{2}\big)\Big\},\\
&\gamma_{2}:=\min\Big\{\|B^{\dagger}\|_{2}^{2}\big(\|EB^{\dagger}B\|_{F}^{2}-\|A\widetilde{E}B\|_{F}^{2}\big),\|B\|_{2}^{2}\big(\|B^{\dagger}B\widetilde{E}\|_{F}^{2}-\|B^{\dagger}EA^{\dagger}\|_{F}^{2}\big)\Big\}.
\end{align*}
Then
\begin{equation}\label{up3.1}
\|P_{B}-P_{A}\|_{F}^{2}\leq\gamma_{1}+\gamma_{2}.
\end{equation}
In particular, if $s=r$, then
\begin{equation}\label{up3.2}
\|P_{B}-P_{A}\|_{F}^{2}\leq 2\min\big\{\gamma_{1},\gamma_{2}\big\}.
\end{equation}
\end{corollary}

\begin{proof}
By~\eqref{SVD-A}, \eqref{SVD-B}, \eqref{MP-A}, and~\eqref{MP-B}, we have
\begin{align}
&U^{\ast}EB^{\dagger}B\widetilde{V}=\begin{pmatrix}
U_{1}^{\ast}\widetilde{U}_{1}\widetilde{\Sigma}_{1}-\Sigma_{1}V_{1}^{\ast}\widetilde{V}_{1} & 0 \\
U_{2}^{\ast}\widetilde{U}_{1}\widetilde{\Sigma}_{1} & 0
\end{pmatrix},\label{EB+B}\\
&\widetilde{V}^{\ast}B^{\dagger}B\widetilde{E}U=\begin{pmatrix}
\widetilde{\Sigma}_{1}^{-1}\widetilde{U}_{1}^{\ast}U_{1}-\widetilde{V}_{1}^{\ast}V_{1}\Sigma_{1}^{-1} & \widetilde{\Sigma}_{1}^{-1}\widetilde{U}_{1}^{\ast}U_{2}\\
0 & 0
\end{pmatrix}.\label{B+BE}
\end{align}
According to~\eqref{AE+B} and~\eqref{EB+B}, we deduce that
\begin{displaymath}
\|\widetilde{U}_{1}^{\ast}U_{2}\|_{F}^{2}\leq\|B^{\dagger}\|_{2}^{2}\|\widetilde{\Sigma}_{1}\widetilde{U}_{1}^{\ast}U_{2}\|_{F}^{2}=\|B^{\dagger}\|_{2}^{2}\big(\|EB^{\dagger}B\|_{F}^{2}-\|A\widetilde{E}B\|_{F}^{2}\big).
\end{displaymath}
On the other hand, we get from~\eqref{B+EA+} and~\eqref{B+BE} that
\begin{displaymath}
\|\widetilde{U}_{1}^{\ast}U_{2}\|_{F}^{2}\leq\|B\|_{2}^{2}\|\widetilde{\Sigma}_{1}^{-1}\widetilde{U}_{1}^{\ast}U_{2}\|_{F}^{2}=\|B\|_{2}^{2}\big(\|B^{\dagger}B\widetilde{E}\|_{F}^{2}-\|B^{\dagger}EA^{\dagger}\|_{F}^{2}\big).
\end{displaymath}
Hence,
\begin{equation}\label{U1starU2-2}
\|\widetilde{U}_{1}^{\ast}U_{2}\|_{F}^{2}\leq\min\Big\{\|B^{\dagger}\|_{2}^{2}\big(\|EB^{\dagger}B\|_{F}^{2}-\|A\widetilde{E}B\|_{F}^{2}\big),\|B\|_{2}^{2}\big(\|B^{\dagger}B\widetilde{E}\|_{F}^{2}-\|B^{\dagger}EA^{\dagger}\|_{F}^{2}\big)\Big\}.
\end{equation}

Similarly, we have
\begin{align}
&\widetilde{U}^{\ast}EA^{\dagger}AV=\begin{pmatrix}
\widetilde{\Sigma}_{1}\widetilde{V}_{1}^{\ast}V_{1}-\widetilde{U}_{1}^{\ast}U_{1}\Sigma_{1} & 0 \\
-\widetilde{U}_{2}^{\ast}U_{1}\Sigma_{1} & 0
\end{pmatrix},\label{EA+A}\\
&V^{\ast}A^{\dagger}A\widetilde{E}\widetilde{U}=\begin{pmatrix}
V_{1}^{\ast}\widetilde{V}_{1}\widetilde{\Sigma}_{1}^{-1}-\Sigma_{1}^{-1}U_{1}^{\ast}\widetilde{U}_{1} & -\Sigma_{1}^{-1}U_{1}^{\ast}\widetilde{U}_{2} \\
0 & 0
\end{pmatrix}.\label{A+AE}
\end{align}
Using~\eqref{BE+A} and~\eqref{EA+A}, we obtain
\begin{displaymath}
\|\widetilde{U}_{2}^{\ast}U_{1}\|_{F}^{2}\leq\|A^{\dagger}\|_{2}^{2}\|\widetilde{U}_{2}^{\ast}U_{1}\Sigma_{1}\|_{F}^{2}=\|A^{\dagger}\|_{2}^{2}\big(\|EA^{\dagger}A\|_{F}^{2}-\|B\widetilde{E}A\|_{F}^{2}\big).
\end{displaymath}
In view of~\eqref{A+EB+} and~\eqref{A+AE}, we have
\begin{displaymath}
\|\widetilde{U}_{2}^{\ast}U_{1}\|_{F}^{2}\leq\|A\|_{2}^{2}\|\widetilde{U}_{2}^{\ast}U_{1}\Sigma_{1}^{-1}\|_{F}^{2}=\|A\|_{2}^{2}\big(\|A^{\dagger}A\widetilde{E}\|_{F}^{2}-\|A^{\dagger}EB^{\dagger}\|_{F}^{2}\big).
\end{displaymath}
Thus,
\begin{equation}\label{U2starU1-2}
\|\widetilde{U}_{2}^{\ast}U_{1}\|_{F}^{2}\leq\min\Big\{\|A^{\dagger}\|_{2}^{2}\big(\|EA^{\dagger}A\|_{F}^{2}-\|B\widetilde{E}A\|_{F}^{2}\big),\|A\|_{2}^{2}\big(\|A^{\dagger}A\widetilde{E}\|_{F}^{2}-\|A^{\dagger}EB^{\dagger}\|_{F}^{2}\big)\Big\}.
\end{equation}
The rest of the proof is similar to Theorem~\ref{up-thm2}.
\end{proof}

\subsection{Lower bounds}

\label{subsec:low}

As is well known, the $L^{2}$-orthogonal projection onto the column space of a matrix is not necessarily a continuous function of the entries of the matrix (see, e.g.,~\cite{Sun1984,Sun2001}). In this subsection, we attempt to establish some lower bounds for $\|P_{B}-P_{A}\|_{F}^{2}$.

The first theorem is based on the identities~\eqref{ide1.1} and~\eqref{ide2.1}.

\begin{theorem}\label{low-thm1}
Let $A\in\mathbb{C}^{m\times n}_{r}$, $B\in\mathbb{C}^{m\times n}_{s}$, $E=B-A$, and $\widetilde{E}=B^{\dagger}-A^{\dagger}$. Define
\begin{align*}
&\alpha'_{1}:=\min\big\{\|B\|_{2}^{2}\|B^{\dagger}EA^{\dagger}\|_{F}^{2},\|A^{\dagger}\|_{2}^{2}\|B\widetilde{E}A\|_{F}^{2}\big\},\\ &\alpha'_{2}:=\min\big\{\|A\|_{2}^{2}\|A^{\dagger}EB^{\dagger}\|_{F}^{2},\|B^{\dagger}\|_{2}^{2}\|A\widetilde{E}B\|_{F}^{2}\big\}.
\end{align*}
Then
\begin{equation}\label{low1.1}
\|P_{B}-P_{A}\|_{F}^{2}\geq\|EA^{\dagger}\|_{F}^{2}+\|EB^{\dagger}\|_{F}^{2}-\alpha'_{1}-\alpha'_{2}.
\end{equation}
In particular, if $s=r$, then
\begin{equation}\label{low1.2}
\|P_{B}-P_{A}\|_{F}^{2}\geq 2\max\big\{\|EA^{\dagger}\|_{F}^{2}-\alpha'_{1},\|EB^{\dagger}\|_{F}^{2}-\alpha'_{2}\big\}.
\end{equation}
\end{theorem}

\begin{proof}
According to the proof of Theorem~\ref{up-thm1}, we have
\begin{align*}
&\|BB^{\dagger}EA^{\dagger}\|_{F}^{2}=\|\widetilde{\Sigma}_{1}(\widetilde{V}_{1}^{\ast}V_{1}\Sigma_{1}^{-1}-\widetilde{\Sigma}_{1}^{-1}\widetilde{U}_{1}^{\ast}U_{1})\|_{F}^{2}\leq\|B\|_{2}^{2}\|B^{\dagger}EA^{\dagger}\|_{F}^{2},\\
&\|BB^{\dagger}EA^{\dagger}\|_{F}^{2}=\|(\widetilde{\Sigma}_{1}\widetilde{V}_{1}^{\ast}V_{1}-\widetilde{U}_{1}^{\ast}U_{1}\Sigma_{1})\Sigma_{1}^{-1}\|_{F}^{2}\leq\|A^{\dagger}\|_{2}^{2}\|B\widetilde{E}A\|_{F}^{2}.
\end{align*}
Hence,
\begin{displaymath}
\|BB^{\dagger}EA^{\dagger}\|_{F}^{2}\leq\min\big\{\|B\|_{2}^{2}\|B^{\dagger}EA^{\dagger}\|_{F}^{2},\|A^{\dagger}\|_{2}^{2}\|B\widetilde{E}A\|_{F}^{2}\big\}.
\end{displaymath}
Similarly, it is easy to check that
\begin{displaymath}
\|AA^{\dagger}EB^{\dagger}\|_{F}^{2}\leq\min\big\{\|A\|_{2}^{2}\|A^{\dagger}EB^{\dagger}\|_{F}^{2},\|B^{\dagger}\|_{2}^{2}\|A\widetilde{E}B\|_{F}^{2}\big\}.
\end{displaymath}
The desired result then follows from the identities~\eqref{ide1.1} and~\eqref{ide2.1}.
\end{proof}

The following theorem is derived by bounding $\|\widetilde{U}_{1}^{\ast}U_{2}\|_{F}^{2}$ and $\|\widetilde{U}_{2}^{\ast}U_{1}\|_{F}^{2}$ directly.

\begin{theorem}\label{low-thm2}
Let $A\in\mathbb{C}^{m\times n}_{r}$, $B\in\mathbb{C}^{m\times n}_{s}$, $E=B-A$, and $\widetilde{E}=B^{\dagger}-A^{\dagger}$. Define
\begin{align*}
&\beta'_{1}:=\max\bigg\{\frac{\|E\|_{F}^{2}-\|AA^{\dagger}E\|_{F}^{2}}{\|B\|_{2}^{2}},\frac{\|\widetilde{E}\|_{F}^{2}-\|\widetilde{E}AA^{\dagger}\|_{F}^{2}}{\|B^{\dagger}\|_{2}^{2}}\bigg\},\\
&\beta'_{2}:=\max\bigg\{\frac{\|E\|_{F}^{2}-\|BB^{\dagger}E\|_{F}^{2}}{\|A\|_{2}^{2}},\frac{\|\widetilde{E}\|_{F}^{2}-\|\widetilde{E}BB^{\dagger}\|_{F}^{2}}{\|A^{\dagger}\|_{2}^{2}}\bigg\}.
\end{align*}
Then
\begin{equation}\label{low2.1}
\|P_{B}-P_{A}\|_{F}^{2}\geq\beta'_{1}+\beta'_{2}.
\end{equation}
In particular, if $s=r$, then
\begin{equation}\label{low2.2}
\|P_{B}-P_{A}\|_{F}^{2}\geq 2\max\big\{\beta'_{1},\beta'_{2}\big\}.
\end{equation}
\end{theorem}

\begin{proof}
Based on the proof of Theorem~\ref{up-thm2}, we have
\begin{align*}
&\|\widetilde{U}_{1}^{\ast}U_{2}\|_{F}^{2}\geq\frac{\|\widetilde{\Sigma}_{1}\widetilde{U}_{1}^{\ast}U_{2}\|_{F}^{2}}{\|B\|_{2}^{2}}=\frac{\|E\|_{F}^{2}-\|AA^{\dagger}E\|_{F}^{2}}{\|B\|_{2}^{2}},\\
&\|\widetilde{U}_{1}^{\ast}U_{2}\|_{F}^{2}\geq\frac{\|\widetilde{\Sigma}_{1}^{-1}\widetilde{U}_{1}^{\ast}U_{2}\|_{F}^{2}}{\|B^{\dagger}\|_{2}^{2}}=\frac{\|\widetilde{E}\|_{F}^{2}-\|\widetilde{E}AA^{\dagger}\|_{F}^{2}}{\|B^{\dagger}\|_{2}^{2}}.
\end{align*}
Thus,
\begin{displaymath}
\|\widetilde{U}_{1}^{\ast}U_{2}\|_{F}^{2}\geq\max\bigg\{\frac{\|E\|_{F}^{2}-\|AA^{\dagger}E\|_{F}^{2}}{\|B\|_{2}^{2}},\frac{\|\widetilde{E}\|_{F}^{2}-\|\widetilde{E}AA^{\dagger}\|_{F}^{2}}{\|B^{\dagger}\|_{2}^{2}}\bigg\}.
\end{displaymath}
Analogously, we have
\begin{align*}
&\|\widetilde{U}_{2}^{\ast}U_{1}\|_{F}^{2}\geq\frac{\|\widetilde{U}_{2}^{\ast}U_{1}\Sigma_{1}\|_{F}^{2}}{\|A\|_{2}^{2}}=\frac{\|E\|_{F}^{2}-\|BB^{\dagger}E\|_{F}^{2}}{\|A\|_{2}^{2}},\\
&\|\widetilde{U}_{2}^{\ast}U_{1}\|_{F}^{2}\geq\frac{\|\widetilde{U}_{2}^{\ast}U_{1}\Sigma_{1}^{-1}\|_{F}^{2}}{\|A^{\dagger}\|_{2}^{2}}=\frac{\|\widetilde{E}\|_{F}^{2}-\|\widetilde{E}BB^{\dagger}\|_{F}^{2}}{\|A^{\dagger}\|_{2}^{2}}.
\end{align*}
Hence,
\begin{displaymath}
\|\widetilde{U}_{2}^{\ast}U_{1}\|_{F}^{2}\geq\max\bigg\{\frac{\|E\|_{F}^{2}-\|BB^{\dagger}E\|_{F}^{2}}{\|A\|_{2}^{2}},\frac{\|\widetilde{E}\|_{F}^{2}-\|\widetilde{E}BB^{\dagger}\|_{F}^{2}}{\|A^{\dagger}\|_{2}^{2}}\bigg\}.
\end{displaymath}
Using~\eqref{exp1.1} and~\eqref{exp2.1}, we can obtain the estimates~\eqref{low2.1} and~\eqref{low2.2}.
\end{proof}

Using the similar argument as in Corollary~\ref{up-cor}, we can get the following corollary, which is an alternative version of Theorem~\ref{low-thm2}.

\begin{corollary}
Let $A\in\mathbb{C}^{m\times n}_{r}$, $B\in\mathbb{C}^{m\times n}_{s}$, $E=B-A$, and $\widetilde{E}=B^{\dagger}-A^{\dagger}$. Define
\begin{align*}
&\gamma'_{1}:=\max\bigg\{\frac{\|EB^{\dagger}B\|_{F}^{2}-\|A\widetilde{E}B\|_{F}^{2}}{\|B\|_{2}^{2}},\frac{\|B^{\dagger}B\widetilde{E}\|_{F}^{2}-\|B^{\dagger}EA^{\dagger}\|_{F}^{2}}{\|B^{\dagger}\|_{2}^{2}}\bigg\},\\
&\gamma'_{2}:=\max\bigg\{\frac{\|EA^{\dagger}A\|_{F}^{2}-\|B\widetilde{E}A\|_{F}^{2}}{\|A\|_{2}^{2}},\frac{\|A^{\dagger}A\widetilde{E}\|_{F}^{2}-\|A^{\dagger}EB^{\dagger}\|_{F}^{2}}{\|A^{\dagger}\|_{2}^{2}}\bigg\}.
\end{align*}
Then
\begin{equation}\label{low3.1}
\|P_{B}-P_{A}\|_{F}^{2}\geq\gamma'_{1}+\gamma'_{2}.
\end{equation}
In particular, if $s=r$, then
\begin{equation}\label{low3.2}
\|P_{B}-P_{A}\|_{F}^{2}\geq 2\max\big\{\gamma'_{1},\gamma'_{2}\big\}.
\end{equation}
\end{corollary}

\subsection{Combined upper bounds}

\label{subsec:comup}

In this subsection, we present new combined upper bounds for $\|P_{B}-P_{A}\|_{F}^{2}$ and $\|P_{B^{\ast}}-P_{A^{\ast}}\|_{F}^{2}$, which are established in a parameterized manner. In order to show the combined upper bounds concisely, we first define
\begin{displaymath}
I_{M}(t):=\frac{t}{\|M^{\dagger}\|_{2}^{2}}+\frac{1-t}{\|M\|_{2}^{2}} \qquad \forall\, M\in\mathbb{C}^{m\times n}\backslash\{0\},\ t\in[0,1].
\end{displaymath}

\begin{theorem}\label{comup-thm}
Let $A\in\mathbb{C}^{m\times n}_{r}$, $B\in\mathbb{C}^{m\times n}_{s}$, $E=B-A$, and $\widetilde{E}=B^{\dagger}-A^{\dagger}$. Define
\begin{align*}
&\Phi(\lambda):=\lambda\big(\|E\|_{F}^{2}-\|A\widetilde{E}B\|_{F}^{2}\big)+(1-\lambda)\big(\|\widetilde{E}\|_{F}^{2}-\|B^{\dagger}EA^{\dagger}\|_{F}^{2}\big),\\
&\Psi(\mu):=\mu\big(\|E\|_{F}^{2}-\|B\widetilde{E}A\|_{F}^{2}\big)+(1-\mu)\big(\|\widetilde{E}\|_{F}^{2}-\|A^{\dagger}EB^{\dagger}\|_{F}^{2}\big),
\end{align*}
where $\lambda\in[0,1]$ and $\mu\in[0,1]$ are parameters. Then
\begin{align}
&\|P_{B}-P_{A}\|_{F}^{2}+\min\bigg\{\frac{I_{A}(\lambda)}{I_{B}(\lambda)},\frac{I_{B}(\mu)}{I_{A}(\mu)}\bigg\}\|P_{B^{\ast}}-P_{A^{\ast}}\|_{F}^{2}\leq\frac{\Phi(\lambda)}{I_{B}(\lambda)}+\frac{\Psi(\mu)}{I_{A}(\mu)}, \label{comup1.1}\\
&\|P_{B}-P_{A}\|_{F}^{2}+\|P_{B^{\ast}}-P_{A^{\ast}}\|_{F}^{2}\leq\frac{\Phi(\lambda)+\Psi(\mu)}{\min\big\{I_{A}(\lambda),I_{B}(\lambda),I_{A}(\mu),I_{B}(\mu)\big\}}.\label{comup1.2}
\end{align}
In particular, if $s=r$, then
\begin{align}
&I_{B}(\lambda)\|P_{B}-P_{A}\|_{F}^{2}+I_{A}(\lambda)\|P_{B^{\ast}}-P_{A^{\ast}}\|_{F}^{2}\leq 2\Phi(\lambda),\label{comup2.1}\\
&I_{A}(\mu)\|P_{B}-P_{A}\|_{F}^{2}+I_{B}(\mu)\|P_{B^{\ast}}-P_{A^{\ast}}\|_{F}^{2}\leq 2\Psi(\mu).\label{comup2.2}
\end{align}
\end{theorem}

\begin{proof}
Using~\eqref{E1} and~\eqref{AE+B}, we obtain
\begin{align*}
\|E\|_{F}^{2}&=\|U_{1}^{\ast}\widetilde{U}_{1}\widetilde{\Sigma}_{1}-\Sigma_{1}V_{1}^{\ast}\widetilde{V}_{1}\|_{F}^{2}+\|\Sigma_{1}V_{1}^{\ast}\widetilde{V}_{2}\|_{F}^{2}+\|U_{2}^{\ast}\widetilde{U}_{1}\widetilde{\Sigma}_{1}\|_{F}^{2}\\
&\geq\|A\widetilde{E}B\|_{F}^{2}+\frac{\|V_{1}^{\ast}\widetilde{V}_{2}\|_{F}^{2}}{\|A^{\dagger}\|_{2}^{2}}+\frac{\|U_{2}^{\ast}\widetilde{U}_{1}\|_{F}^{2}}{\|B^{\dagger}\|_{2}^{2}},
\end{align*}
which gives
\begin{equation}\label{comrela1.1}
\frac{\|\widetilde{U}_{1}^{\ast}U_{2}\|_{F}^{2}}{\|B^{\dagger}\|_{2}^{2}}+\frac{\|\widetilde{V}_{2}^{\ast}V_{1}\|_{F}^{2}}{\|A^{\dagger}\|_{2}^{2}}\leq\|E\|_{F}^{2}-\|A\widetilde{E}B\|_{F}^{2}.
\end{equation}
By~\eqref{E2} and~\eqref{BE+A}, we have
\begin{align*}
\|E\|_{F}^{2}&=\|\widetilde{\Sigma}_{1}\widetilde{V}_{1}^{\ast}V_{1}-\widetilde{U}_{1}^{\ast}U_{1}\Sigma_{1}\|_{F}^{2}+\|\widetilde{\Sigma}_{1}\widetilde{V}_{1}^{\ast}V_{2}\|_{F}^{2}+\|\widetilde{U}_{2}^{\ast}U_{1}\Sigma_{1}\|_{F}^{2}\\
&\geq\|B\widetilde{E}A\|_{F}^{2}+\frac{\|\widetilde{V}_{1}^{\ast}V_{2}\|_{F}^{2}}{\|B^{\dagger}\|_{2}^{2}}+\frac{\|\widetilde{U}_{2}^{\ast}U_{1}\|_{F}^{2}}{\|A^{\dagger}\|_{2}^{2}},
\end{align*}
which yields
\begin{equation}\label{comrela1.2}
\frac{\|\widetilde{U}_{2}^{\ast}U_{1}\|_{F}^{2}}{\|A^{\dagger}\|_{2}^{2}}+\frac{\|\widetilde{V}_{1}^{\ast}V_{2}\|_{F}^{2}}{\|B^{\dagger}\|_{2}^{2}}\leq\|E\|_{F}^{2}-\|B\widetilde{E}A\|_{F}^{2}.
\end{equation}
Similarly, we can derive from~\eqref{E+1}, \eqref{E+2}, \eqref{B+EA+}, and~\eqref{A+EB+} that
\begin{align}
&\frac{\|\widetilde{U}_{1}^{\ast}U_{2}\|_{F}^{2}}{\|B\|_{2}^{2}}+\frac{\|\widetilde{V}_{2}^{\ast}V_{1}\|_{F}^{2}}{\|A\|_{2}^{2}}\leq\|\widetilde{E}\|_{F}^{2}-\|B^{\dagger}EA^{\dagger}\|_{F}^{2},\label{comrela1.3}\\
&\frac{\|\widetilde{U}_{2}^{\ast}U_{1}\|_{F}^{2}}{\|A\|_{2}^{2}}+\frac{\|\widetilde{V}_{1}^{\ast}V_{2}\|_{F}^{2}}{\|B\|_{2}^{2}}\leq\|\widetilde{E}\|_{F}^{2}-\|A^{\dagger}EB^{\dagger}\|_{F}^{2}.\label{comrela1.4}
\end{align}
From~\eqref{comrela1.1} and~\eqref{comrela1.3}, we deduce that
\begin{equation}\label{compart1.1}
I_{B}(\lambda)\|\widetilde{U}_{1}^{\ast}U_{2}\|_{F}^{2}+I_{A}(\lambda)\|\widetilde{V}_{2}^{\ast}V_{1}\|_{F}^{2}\leq\Phi(\lambda).
\end{equation}
In light of~\eqref{comrela1.2} and~\eqref{comrela1.4}, we have
\begin{equation}\label{compart1.2}
I_{A}(\mu)\|\widetilde{U}_{2}^{\ast}U_{1}\|_{F}^{2}+I_{B}(\mu)\|\widetilde{V}_{1}^{\ast}V_{2}\|_{F}^{2}\leq\Psi(\mu).
\end{equation}

Combining~\eqref{exp1.1}, \eqref{exp1.2}, \eqref{compart1.1}, and~\eqref{compart1.2}, we can arrive at the estimates~\eqref{comup1.1} and~\eqref{comup1.2}. In particular, if $s=r$, using~\eqref{exp2.1}, \eqref{exp2.2}, \eqref{compart1.1}, and~\eqref{compart1.2}, we can obtain the estimates~\eqref{comup2.1} and~\eqref{comup2.2}.
\end{proof}

Under the assumptions of Theorem~\ref{comup-thm}, taking $\lambda=\mu=1$, we can get the following corollary.

\begin{corollary}
Let $A\in\mathbb{C}^{m\times n}_{r}$, $B\in\mathbb{C}^{m\times n}_{s}$, $E=B-A$, and $\widetilde{E}=B^{\dagger}-A^{\dagger}$. Then
\begin{align}
\|P_{B}-P_{A}\|_{F}^{2}&+\min\bigg\{\frac{\|A^{\dagger}\|_{2}^{2}}{\|B^{\dagger}\|_{2}^{2}},\frac{\|B^{\dagger}\|_{2}^{2}}{\|A^{\dagger}\|_{2}^{2}}\bigg\}\|P_{B^{\ast}}-P_{A^{\ast}}\|_{F}^{2}\notag\\
&\quad\leq\big(\|A^{\dagger}\|_{2}^{2}+\|B^{\dagger}\|_{2}^{2}\big)\|E\|_{F}^{2}-\|B^{\dagger}\|_{2}^{2}\|A\widetilde{E}B\|_{F}^{2}-\|A^{\dagger}\|_{2}^{2}\|B\widetilde{E}A\|_{F}^{2},\label{corup1.1}\\
\|P_{B}-P_{A}\|_{F}^{2}&+\|P_{B^{\ast}}-P_{A^{\ast}}\|_{F}^{2}\leq\max\big\{\|A^{\dagger}\|_{2}^{2},\|B^{\dagger}\|_{2}^{2}\big\}\big(2\|E\|_{F}^{2}-\|A\widetilde{E}B\|_{F}^{2}-\|B\widetilde{E}A\|_{F}^{2}\big).\label{corup1.2}
\end{align}
In particular, if $s=r$, then
\begin{align}
\|P_{B}-P_{A}\|_{F}^{2}&+\min\bigg\{\frac{\|A^{\dagger}\|_{2}^{2}}{\|B^{\dagger}\|_{2}^{2}},\frac{\|B^{\dagger}\|_{2}^{2}}{\|A^{\dagger}\|_{2}^{2}}\bigg\}\|P_{B^{\ast}}-P_{A^{\ast}}\|_{F}^{2}\notag\\
&\quad\leq 2\min\big\{\|A^{\dagger}\|_{2}^{2}\|E\|_{F}^{2}-\|A^{\dagger}\|_{2}^{2}\|B\widetilde{E}A\|_{F}^{2},\|B^{\dagger}\|_{2}^{2}\|E\|_{F}^{2}-\|B^{\dagger}\|_{2}^{2}\|A\widetilde{E}B\|_{F}^{2}\big\},\label{corup2.1}\\
\|P_{B}-P_{A}\|_{F}^{2}&+\|P_{B^{\ast}}-P_{A^{\ast}}\|_{F}^{2}\leq\frac{2\|A^{\dagger}\|_{2}^{2}\|B^{\dagger}\|_{2}^{2}}{\|A^{\dagger}\|_{2}^{2}+\|B^{\dagger}\|_{2}^{2}}\big(2\|E\|_{F}^{2}-\|A\widetilde{E}B\|_{F}^{2}-\|B\widetilde{E}A\|_{F}^{2}\big).\label{corup2.2}
\end{align}
\end{corollary}

\begin{remark}\rm
Evidently, the estimates~\eqref{corup1.1}, \eqref{corup2.1}, and~\eqref{corup2.2} are sharper than~\eqref{Chen-comb1}, \eqref{Chen-comb2}, and~\eqref{Chen-comb3}, respectively. In addition, since
\begin{align*}
&\|A\widetilde{E}B\|_{F}^{2}=\|(U_{1}^{\ast}\widetilde{U}_{1}-\Sigma_{1}V_{1}^{\ast}\widetilde{V}_{1}\widetilde{\Sigma}_{1}^{-1})\widetilde{\Sigma}_{1}\|_{F}^{2}\geq\frac{\|AA^{\dagger}EB^{\dagger}\|_{F}^{2}}{\|B^{\dagger}\|_{2}^{2}}\geq\frac{\|A^{\dagger}EB^{\dagger}\|_{F}^{2}}{\|A^{\dagger}\|_{2}^{2}\|B^{\dagger}\|_{2}^{2}},\\
&\|B\widetilde{E}A\|_{F}^{2}=\|(\widetilde{\Sigma}_{1}\widetilde{V}_{1}^{\ast}V_{1}\Sigma_{1}^{-1}-\widetilde{U}_{1}^{\ast}U_{1})\Sigma_{1}\|_{F}^{2}\geq\frac{\|BB^{\dagger}EA^{\dagger}\|_{F}^{2}}{\|A^{\dagger}\|_{2}^{2}}\geq\frac{\|B^{\dagger}EA^{\dagger}\|_{F}^{2}}{\|A^{\dagger}\|_{2}^{2}\|B^{\dagger}\|_{2}^{2}},
\end{align*}
we conclude that~\eqref{corup1.2} and~\eqref{corup2.2} are sharper than~\eqref{Li-comb1} and~\eqref{Li-comb2}, respectively.
\end{remark}

\subsection{Combined lower bounds}

\label{subsec:comlow}

In this subsection, we develop some combined lower bounds for $\|P_{B}-P_{A}\|_{F}^{2}$ and $\|P_{B^{\ast}}-P_{A^{\ast}}\|_{F}^{2}$. For simplicity, we define
\begin{displaymath}
J_{M}(t):=t\|M\|_{2}^{2}+(1-t)\|M^{\dagger}\|_{2}^{2} \qquad \forall\, M\in\mathbb{C}^{m\times n}\backslash\{0\},\ t\in[0,1].
\end{displaymath}

\begin{theorem}\label{comlow-thm}
Let $A\in\mathbb{C}^{m\times n}_{r}$, $B\in\mathbb{C}^{m\times n}_{s}$, $E=B-A$, and $\widetilde{E}=B^{\dagger}-A^{\dagger}$. Let $\Phi(\xi)$ and $\Psi(\eta)$ be defined as in Theorem~{\rm\ref{comup-thm}}, where $\xi\in[0,1]$ and $\eta\in[0,1]$ are parameters. Then
\begin{align}
&\|P_{B}-P_{A}\|_{F}^{2}+\max\bigg\{\frac{J_{A}(\xi)}{J_{B}(\xi)},\frac{J_{B}(\eta)}{J_{A}(\eta)}\bigg\}\|P_{B^{\ast}}-P_{A^{\ast}}\|_{F}^{2}\geq\frac{\Phi(\xi)}{J_{B}(\xi)}+\frac{\Psi(\eta)}{J_{A}(\eta)}, \label{comlow1.1}\\
&\|P_{B}-P_{A}\|_{F}^{2}+\|P_{B^{\ast}}-P_{A^{\ast}}\|_{F}^{2}\geq\frac{\Phi(\xi)+\Psi(\eta)}{\max\big\{J_{A}(\xi),J_{B}(\xi),J_{A}(\eta),J_{B}(\eta)\big\}}.\label{comlow1.2}
\end{align}
In particular, if $s=r$, then
\begin{align}
&J_{B}(\xi)\|P_{B}-P_{A}\|_{F}^{2}+J_{A}(\xi)\|P_{B^{\ast}}-P_{A^{\ast}}\|_{F}^{2}\geq 2\Phi(\xi),\label{comlow2.1}\\
&J_{A}(\eta)\|P_{B}-P_{A}\|_{F}^{2}+J_{B}(\eta)\|P_{B^{\ast}}-P_{A^{\ast}}\|_{F}^{2}\geq 2\Psi(\eta).\label{comlow2.2}
\end{align}
\end{theorem}

\begin{proof}
According to the proof of Theorem~\ref{comup-thm}, we deduce that
\begin{align}
&\|B\|_{2}^{2}\|\widetilde{U}_{1}^{\ast}U_{2}\|_{F}^{2}+\|A\|_{2}^{2}\|\widetilde{V}_{2}^{\ast}V_{1}\|_{F}^{2}\geq\|E\|_{F}^{2}-\|A\widetilde{E}B\|_{F}^{2},\label{comrela2.1}\\
&\|A\|_{2}^{2}\|\widetilde{U}_{2}^{\ast}U_{1}\|_{F}^{2}+\|B\|_{2}^{2}\|\widetilde{V}_{1}^{\ast}V_{2}\|_{F}^{2}\geq\|E\|_{F}^{2}-\|B\widetilde{E}A\|_{F}^{2},\label{comrela2.2}\\
&\|B^{\dagger}\|_{2}^{2}\|\widetilde{U}_{1}^{\ast}U_{2}\|_{F}^{2}+\|A^{\dagger}\|_{2}^{2}\|\widetilde{V}_{2}^{\ast}V_{1}\|_{F}^{2}\geq\|\widetilde{E}\|_{F}^{2}-\|B^{\dagger}EA^{\dagger}\|_{F}^{2},\label{comrela2.3}\\
&\|A^{\dagger}\|_{2}^{2}\|\widetilde{U}_{2}^{\ast}U_{1}\|_{F}^{2}+\|B^{\dagger}\|_{2}^{2}\|\widetilde{V}_{1}^{\ast}V_{2}\|_{F}^{2}\geq\|\widetilde{E}\|_{F}^{2}-\|A^{\dagger}EB^{\dagger}\|_{F}^{2}.\label{comrela2.4}
\end{align}
Using~\eqref{comrela2.1}, \eqref{comrela2.2}, \eqref{comrela2.3}, and~\eqref{comrela2.4}, we can obtain
\begin{align*}
&J_{B}(\xi)\|\widetilde{U}_{1}^{\ast}U_{2}\|_{F}^{2}+J_{A}(\xi)\|\widetilde{V}_{2}^{\ast}V_{1}\|_{F}^{2}\geq \Phi(\xi),\\
&J_{A}(\eta)\|\widetilde{U}_{2}^{\ast}U_{1}\|_{F}^{2}+J_{B}(\eta)\|\widetilde{V}_{1}^{\ast}V_{2}\|_{F}^{2}\geq \Psi(\eta).
\end{align*}
The rest of the proof is similar to Theorem~\ref{comup-thm}.
\end{proof}

Taking $\xi=\eta=0$, we can obtain the following corollary.

\begin{corollary}
Let $A\in\mathbb{C}^{m\times n}_{r}$, $B\in\mathbb{C}^{m\times n}_{s}$, $E=B-A$, and $\widetilde{E}=B^{\dagger}-A^{\dagger}$. Then
\begin{align}
\|P_{B}-P_{A}\|_{F}^{2}&+\max\bigg\{\frac{\|A^{\dagger}\|_{2}^{2}}{\|B^{\dagger}\|_{2}^{2}},\frac{\|B^{\dagger}\|_{2}^{2}}{\|A^{\dagger}\|_{2}^{2}}\bigg\}\|P_{B^{\ast}}-P_{A^{\ast}}\|_{F}^{2}\notag\\
&\quad\geq\bigg(\frac{1}{\|A^{\dagger}\|_{2}^{2}}+\frac{1}{\|B^{\dagger}\|_{2}^{2}}\bigg)\|\widetilde{E}\|_{F}^{2}-\frac{\|A^{\dagger}EB^{\dagger}\|_{F}^{2}}{\|A^{\dagger}\|_{2}^{2}}-\frac{\|B^{\dagger}EA^{\dagger}\|_{F}^{2}}{\|B^{\dagger}\|_{2}^{2}},\label{corlow1.1}\\
\|P_{B}-P_{A}\|_{F}^{2}&+\|P_{B^{\ast}}-P_{A^{\ast}}\|_{F}^{2}\geq\frac{2\|\widetilde{E}\|_{F}^{2}-\|A^{\dagger}EB^{\dagger}\|_{F}^{2}-\|B^{\dagger}EA^{\dagger}\|_{F}^{2}}{\max\big\{\|A^{\dagger}\|_{2}^{2},\|B^{\dagger}\|_{2}^{2}\big\}}.\label{corlow1.2}
\end{align}
In particular, if $s=r$, then
\begin{align}
\|P_{B}-P_{A}\|_{F}^{2}&+\max\bigg\{\frac{\|A^{\dagger}\|_{2}^{2}}{\|B^{\dagger}\|_{2}^{2}},\frac{\|B^{\dagger}\|_{2}^{2}}{\|A^{\dagger}\|_{2}^{2}}\bigg\}\|P_{B^{\ast}}-P_{A^{\ast}}\|_{F}^{2}\notag\\
&\quad\geq 2\max\bigg\{\frac{\|\widetilde{E}\|_{F}^{2}-\|B^{\dagger}EA^{\dagger}\|_{F}^{2}}{\|B^{\dagger}\|_{2}^{2}},\frac{\|\widetilde{E}\|_{F}^{2}-\|A^{\dagger}EB^{\dagger}\|_{F}^{2}}{\|A^{\dagger}\|_{2}^{2}}\bigg\},\label{corlow2.1}\\
\|P_{B}-P_{A}\|_{F}^{2}&+\|P_{B^{\ast}}-P_{A^{\ast}}\|_{F}^{2}\geq \frac{2}{\|A^{\dagger}\|_{2}^{2}+\|B^{\dagger}\|_{2}^{2}}\big(2\|\widetilde{E}\|_{F}^{2}-\|A^{\dagger}EB^{\dagger}\|_{F}^{2}-\|B^{\dagger}EA^{\dagger}\|_{F}^{2}\big).\label{corlow2.2}
\end{align}
\end{corollary}

\begin{remark}\rm
The parameters $\lambda$, $\mu$, $\xi$, and $\eta$ in Theorems~\ref{comup-thm} and~\ref{comlow-thm} can be chosen flexibly. Different parameters will yield different types of combined estimates. Thus, one can optimize the combined bounds in Theorems~\ref{comup-thm} and~\ref{comlow-thm} by selecting some sophisticated parameters.
\end{remark}

\section{Numerical experiments}

\label{sec:numer}

\setcounter{equation}{0}

In Section~\ref{sec:main}, we have developed new perturbation bounds for the $L^{2}$-orthogonal projection onto the column space of a matrix, and compared the new results with the existing ones theoretically. In this section, we give two examples to illustrate the differences between the new bounds and the existing ones. The first one is in fact the example in~\eqref{Ex0}.

\begin{example}\label{Ex1}\rm
Let
\begin{displaymath}
A=\begin{pmatrix}
1 & 0 \\
0 & 0
\end{pmatrix} \quad \text{and} \quad B=\begin{pmatrix}
\frac{\varepsilon}{1+\varepsilon} & 0 \\
0 & \frac{\varepsilon}{10}
\end{pmatrix},
\end{displaymath}
where $0<\varepsilon<1$.
\end{example}

In this example, we have
\begin{displaymath}
E=\begin{pmatrix}
-\frac{1}{1+\varepsilon} & 0 \\
0 & \frac{\varepsilon}{10}
\end{pmatrix}, \quad A^{\dagger}=\begin{pmatrix}
1 & 0 \\
0 & 0
\end{pmatrix}, \quad B^{\dagger}=\begin{pmatrix}
1+\frac{1}{\varepsilon} & 0 \\
0 & \frac{10}{\varepsilon}
\end{pmatrix}, \quad \widetilde{E}=\begin{pmatrix}
\frac{1}{\varepsilon} & 0 \\
0 & \frac{10}{\varepsilon}
\end{pmatrix}.
\end{displaymath}
It is easy to see that
\begin{displaymath}
\|P_{B}-P_{A}\|_{F}^{2}\equiv 1 \quad \forall\,0<\varepsilon<1.
\end{displaymath}

\medskip

(\uppercase\expandafter{\romannumeral1}) \textit{Upper and lower bounds}

\smallskip

Under the setting of Example~\ref{Ex1}, the upper bounds in~\eqref{Chen1}, \eqref{Li1}, \eqref{rank1.2}, \eqref{up1.1}, and~\eqref{up2.1} are listed in Table~\ref{tab:upper1}. And the numerical behaviors ($\varepsilon$ is confined in $(0.1,1)$) of these bounds are shown in Figure~\ref{fig:upper1}.

\begin{table}[h!!]
\centering
\setlength{\tabcolsep}{12mm}{
\begin{tabular}{@{} cc @{}}
\toprule
\text{Estimate} & \text{Upper bound for $\|P_{B}-P_{A}\|_{F}^{2}$} \\
\midrule
\eqref{Chen1} & $1+\frac{1}{\varepsilon^{2}}+\frac{1}{(1+\varepsilon)^{2}}$ \\
\eqref{Li1} & $\frac{99}{100}+\frac{1}{(1+\varepsilon)^{2}}$ \\
\eqref{rank1.2} & $1$ \\
\eqref{up1.1} & $1$ \\
\eqref{up2.1} & $1$ \\
\bottomrule
\end{tabular}}
\caption{\small The upper bounds in~\eqref{Chen1}, \eqref{Li1}, \eqref{rank1.2}, \eqref{up1.1}, and~\eqref{up2.1}.}
\label{tab:upper1}
\end{table}

\begin{figure}[h!!]
\centering
\includegraphics[width=3.1in]{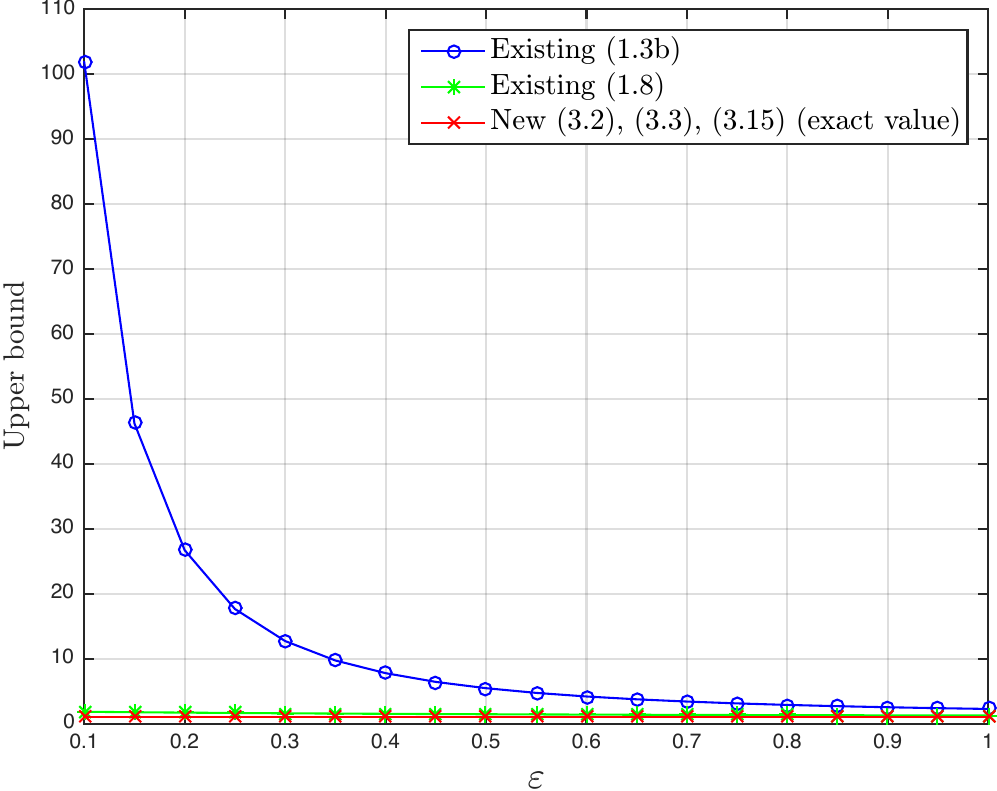}
\caption{\small Numerical comparison of the upper bounds listed in Table~\ref{tab:upper1}.}
\label{fig:upper1}
\end{figure}

From Table~\ref{tab:upper1}, we see that the upper bounds in~\eqref{rank1.2}, \eqref{up1.1}, and~\eqref{up2.1} have attained the exact value $1$. Figure~\ref{fig:upper1} shows that the upper bound in~\eqref{Chen1} will deviate from the exact value seriously when $\varepsilon$ is small.

In addition, direct computations yield that the lower bounds in~\eqref{rank1.2}, \eqref{low1.1}, and~\eqref{low2.1} are all the exact value $1$.

\medskip

(\uppercase\expandafter{\romannumeral2}) \textit{Combined upper and lower bounds}

\smallskip

For simplicity, we define
\begin{align*}
&\mathscr{C}_{1}:=\|P_{B}-P_{A}\|_{F}^{2}+\min\bigg\{\frac{\|A^{\dagger}\|_{2}^{2}}{\|B^{\dagger}\|_{2}^{2}},\frac{\|B^{\dagger}\|_{2}^{2}}{\|A^{\dagger}\|_{2}^{2}}\bigg\}\|P_{B^{\ast}}-P_{A^{\ast}}\|_{F}^{2},\\ &\mathscr{C}_{2}:=\|P_{B}-P_{A}\|_{F}^{2}+\|P_{B^{\ast}}-P_{A^{\ast}}\|_{F}^{2},\\
&\mathscr{C}_{3}:=\|P_{B}-P_{A}\|_{F}^{2}+\max\bigg\{\frac{\|A^{\dagger}\|_{2}^{2}}{\|B^{\dagger}\|_{2}^{2}},\frac{\|B^{\dagger}\|_{2}^{2}}{\|A^{\dagger}\|_{2}^{2}}\bigg\}\|P_{B^{\ast}}-P_{A^{\ast}}\|_{F}^{2}.
\end{align*}

Under the setting of Example~\ref{Ex1}, we have
\begin{displaymath}
\mathscr{C}_{1}=1+\frac{\varepsilon^{2}}{100}, \quad \mathscr{C}_{2}=2, \quad \text{and} \quad \mathscr{C}_{3}=1+\frac{100}{\varepsilon^{2}}.
\end{displaymath}
The combined upper bounds for $\mathscr{C}_{1}$ in~\eqref{Chen-comb1} and~\eqref{corup1.1} are given in Table~\ref{tab:comup1.1}, and the combined upper bounds for $\mathscr{C}_{2}$ in~\eqref{Li-comb1} and~\eqref{corup1.2} are listed in Table~\ref{tab:comup1.2}. The numerical behaviors ($\varepsilon$ is confined in $(0.1,1)$) of these bounds are shown in Figure~\ref{fig:comup1}.

\begin{table}[h!!]
\centering
\setlength{\tabcolsep}{12mm}{
\begin{tabular}{@{} cc @{}}
\toprule
\text{Estimate} & \text{Combined upper bound for $\mathscr{C}_{1}$} \\
\midrule
\eqref{Chen-comb1} & $1+\frac{\varepsilon^{2}}{100}+\frac{1}{(1+\varepsilon)^{2}}+\frac{100}{\varepsilon^{2}(1+\varepsilon)^{2}}$ \\
\eqref{corup1.1} & $1+\frac{\varepsilon^{2}}{100}$ \\
\bottomrule
\end{tabular}}
\caption{\small The combined upper bounds in~\eqref{Chen-comb1} and~\eqref{corup1.1}.}
\label{tab:comup1.1}
\end{table}

\begin{table}[h!!]
\centering
\setlength{\tabcolsep}{12mm}{
\begin{tabular}{@{} cc @{}}
\toprule
\text{Estimate} & \text{Combined upper bound for $\mathscr{C}_{2}$} \\
\midrule
\eqref{Li-comb1} & $2-\frac{2}{\varepsilon^{2}}+\frac{200}{\varepsilon^{2}(1+\varepsilon)^{2}}$ \\
\eqref{corup1.2}  & $2$ \\
\bottomrule
\end{tabular}}
\caption{\small The combined upper bounds in~\eqref{Li-comb1} and~\eqref{corup1.2}.}
\label{tab:comup1.2}
\end{table}

\begin{figure}[h!!]
\centering
\begin{tabular}{cc}
\begin{minipage}[t]{3.1in}
\includegraphics[width=3.02in]{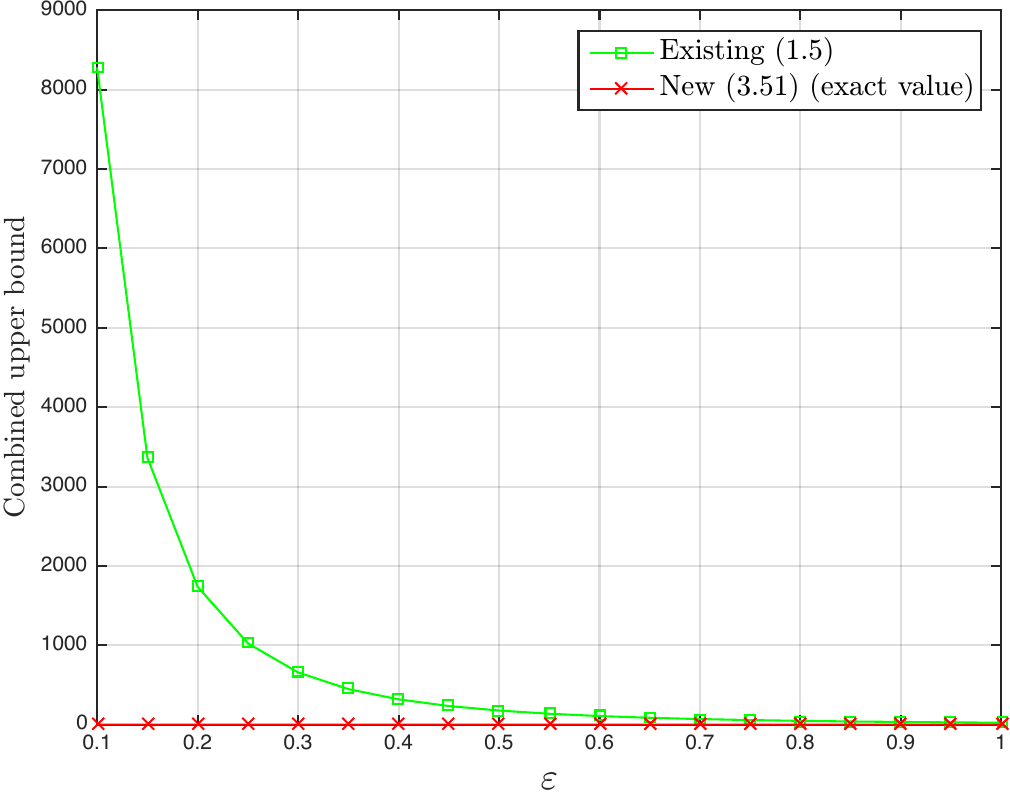}
\end{minipage}
\begin{minipage}[t]{3.1in}
\includegraphics[width=3.06in]{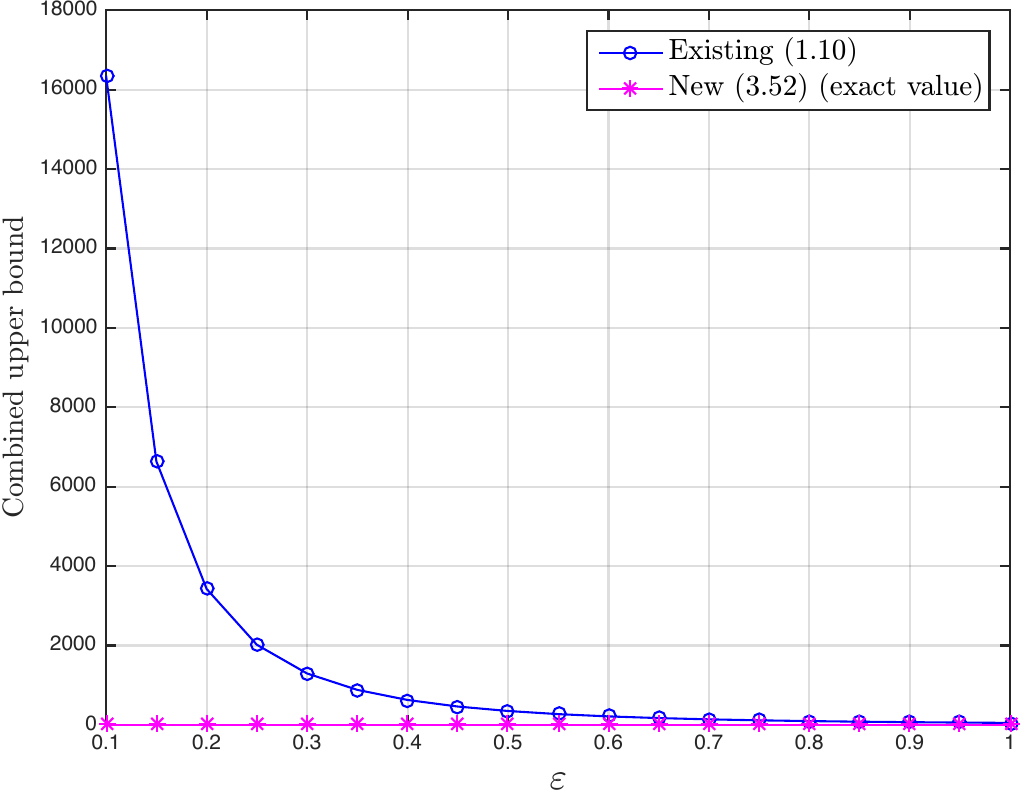}
\end{minipage}
\end{tabular}
\caption{\small Numerical comparison of the combined upper bounds in Table~\ref{tab:comup1.1} (left); numerical comparison of the combined upper bounds in Table~\ref{tab:comup1.2} (right).}
\label{fig:comup1}
\end{figure}

From Tables~\ref{tab:comup1.1} and~\ref{tab:comup1.2}, we see that the combined upper bounds in~\eqref{corup1.1} and~\eqref{corup1.2} have attained the exact values $1+\frac{\varepsilon^{2}}{100}$ and $2$, respectively. Figure~\ref{fig:comup1} displays that the existing bounds in~\eqref{Chen-comb1} and~\eqref{Li-comb1} will deviate from the corresponding exact values seriously when $\varepsilon$ is small.

Furthermore, straightforward calculations yield that the lower bound for $\mathscr{C}_{3}$ in~\eqref{corlow1.1} is $1+\frac{100}{\varepsilon^{2}}$ and the lower bound for $\mathscr{C}_{2}$ in~\eqref{corlow1.2} is $2$. Thus, the combined lower bounds in~\eqref{corlow1.1} and~\eqref{corlow1.2} have attained the corresponding exact values.

The next example provides a complex matrix case.

\begin{example}\label{Ex2}\rm
Let
\begin{displaymath}
A=\begin{pmatrix}
i & 0 \\
0 & 0
\end{pmatrix} \quad \text{and} \quad B=\begin{pmatrix}
\frac{i}{1+\varepsilon} & \varepsilon \\
0 & \varepsilon
\end{pmatrix},
\end{displaymath}
where $i=\sqrt{-1}$ and $0<\varepsilon<\frac{1}{2}$.
\end{example}

In this example, we have
\begin{displaymath}
E=\begin{pmatrix}
-\frac{\varepsilon i}{1+\varepsilon} & \varepsilon \\
0 & \varepsilon
\end{pmatrix}, \quad A^{\dagger}=\begin{pmatrix}
-i & 0 \\
0 & 0
\end{pmatrix}, \quad B^{\dagger}=\begin{pmatrix}
-(1+\varepsilon)i & (1+\varepsilon)i \\
0 & \frac{1}{\varepsilon}
\end{pmatrix}, \quad \widetilde{E}=\begin{pmatrix}
-\varepsilon i & (1+\varepsilon)i \\
0 & \frac{1}{\varepsilon}
\end{pmatrix}.
\end{displaymath}
Obviously, it holds that
\begin{displaymath}
\|P_{B}-P_{A}\|_{F}^{2}\equiv 1 \quad \forall\,0<\varepsilon<\frac{1}{2}.
\end{displaymath}

\medskip

(\uppercase\expandafter{\romannumeral1}) \textit{Upper and lower bounds}

\smallskip

Under the setting of Example~\ref{Ex2}, the upper bounds in~\eqref{Chen1}, \eqref{Li1}, \eqref{rank1.2}, \eqref{up1.1}, and~\eqref{up2.1} are given in Table~\ref{tab:upper2}, and the lower bounds in~\eqref{rank1.2}, \eqref{low1.1}, and~\eqref{low2.1} are listed in Table~\ref{tab:lower2}. Numerical behaviors of these bounds are shown in Figure~\ref{fig:upper2}.

\begin{table}[h!!]
\centering
\setlength{\tabcolsep}{1mm}{
\begin{tabular}{@{} cc @{}}
\toprule
\text{Estimate} & \text{Upper bound for $\|P_{B}-P_{A}\|_{F}^{2}$} \\
\midrule
\eqref{Chen1} & $2(1+\varepsilon+\varepsilon^{2})+\frac{\varepsilon^{2}}{(1+\varepsilon)^{2}}$ \\
\eqref{Li1} & $2\varepsilon^{2}(1+\varepsilon)+\frac{\varepsilon}{1+\varepsilon}+(1+\varepsilon+\varepsilon^{2})\sqrt{4\varepsilon^{4}+\frac{1}{(1+\varepsilon)^{4}}}$ \\
\eqref{rank1.2} & $1$ \\
\eqref{up1.1} & $1$ \\
\eqref{up2.1} & $\frac{1}{2}+\varepsilon^{2}(1+\varepsilon)^{2}+\sqrt{\frac{1}{4}+\varepsilon^{4}(1+\varepsilon)^{4}}$ \\
\bottomrule
\end{tabular}}
\caption{\small The upper bounds in~\eqref{Chen1}, \eqref{Li1}, \eqref{rank1.2}, \eqref{up1.1}, and~\eqref{up2.1}.}
\label{tab:upper2}
\end{table}

\begin{table}[h!!]
\centering
\setlength{\tabcolsep}{14mm}{
\begin{tabular}{@{} cc @{}}
\toprule
\text{Estimate} & \text{Lower bound for $\|P_{B}-P_{A}\|_{F}^{2}$} \\
\midrule
\eqref{rank1.2} & $1$ \\
\eqref{low1.1} & $1$ \\
\eqref{low2.1} & $\frac{2+2\varepsilon^{2}(1+\varepsilon)^{2}}{1+2\varepsilon^{2}(1+\varepsilon)^{2}+\sqrt{1+4\varepsilon^{4}(1+\varepsilon)^{4}}}$ \\
\bottomrule
\end{tabular}}
\caption{\small The lower bounds in~\eqref{rank1.2}, \eqref{low1.1}, and~\eqref{low2.1}.}
\label{tab:lower2}
\end{table}

\begin{figure}[h!!]
\centering
\begin{tabular}{cc}
\begin{minipage}[t]{3.1in}
\includegraphics[width=3.02in]{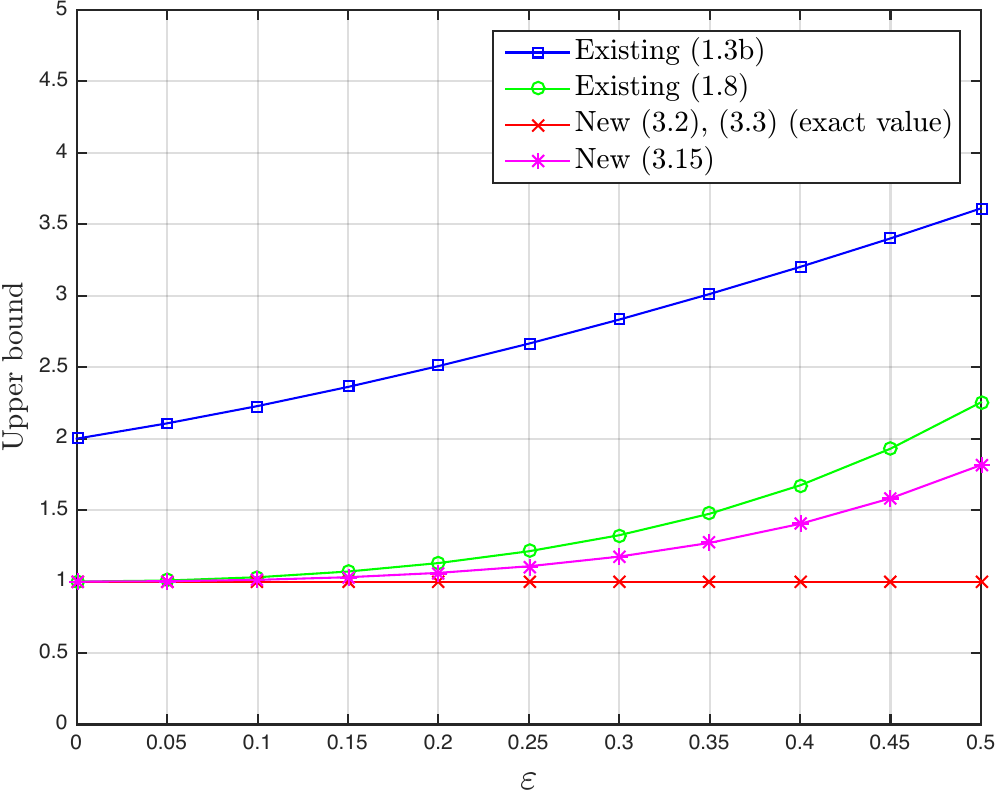}
\end{minipage}
\begin{minipage}[t]{3.1in}
\includegraphics[width=3.06in]{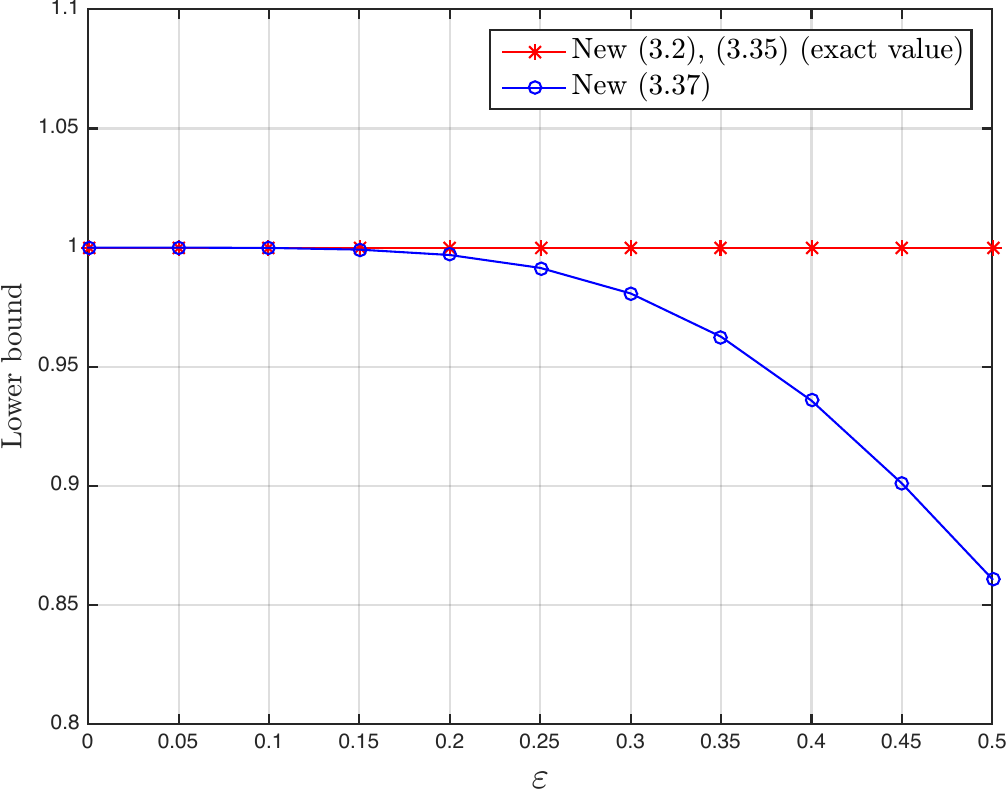}
\end{minipage}
\end{tabular}
\caption{\small Numerical comparison of the upper bounds in Table~\ref{tab:upper2} (left); numerical comparison of the lower bounds in Table~\ref{tab:lower2} (right).}
\label{fig:upper2}
\end{figure}

From Table~\ref{tab:upper2}, we see that the upper bounds in~\eqref{rank1.2} and~\eqref{up1.1} have attained the exact value $1$. And Figure~\ref{fig:upper2} (left) shows that the estimate~\eqref{up2.1} is sharper than both~\eqref{Chen1} and~\eqref{Li1}.

From Table~\ref{tab:lower2}, we see that the lower bounds in~\eqref{rank1.2} and~\eqref{low1.1} have attained the exact value $1$. Moreover, Figure~\ref{fig:upper2} (right) displays that the lower bound in~\eqref{low2.1} is very close to the exact value (especially when $\varepsilon$ is small).

\medskip

(\uppercase\expandafter{\romannumeral2}) \textit{Combined upper and lower bounds}

\smallskip

Under the setting of Example~\ref{Ex2}, we have
\begin{align*}
\mathscr{C}_{1}&=1+\frac{2\varepsilon^{2}}{1+2\varepsilon^{2}(1+\varepsilon)^{2}+\sqrt{1+4\varepsilon^{4}(1+\varepsilon)^{4}}},\\
\mathscr{C}_{2}&=2,\\
\mathscr{C}_{3}&=1+(1+\varepsilon)^{2}+\frac{1}{2\varepsilon^{2}}+\sqrt{(1+\varepsilon)^{4}+\frac{1}{4\varepsilon^{4}}}.
\end{align*}

The combined upper bounds for $\mathscr{C}_{1}$ in~\eqref{Chen-comb1} and~\eqref{corup1.1} are listed in Table~\ref{tab:comup2.1}, and the combined upper bounds for $\mathscr{C}_{2}$ in~\eqref{Li-comb1} and~\eqref{corup1.2} are given in Table~\ref{tab:comup2.2}. In addition, the lower bound for $\mathscr{C}_{3}$ in~\eqref{corlow1.1} is
\begin{displaymath}
\frac{1}{\varepsilon^{2}}+\frac{2+2\varepsilon^{2}(1+\varepsilon)^{2}}{1+2\varepsilon^{2}(1+\varepsilon)^{2}+\sqrt{1+4\varepsilon^{4}(1+\varepsilon)^{4}}},
\end{displaymath}
and the lower bound for $\mathscr{C}_{2}$ in~\eqref{corlow1.2} is
\begin{displaymath}
\frac{4+2\varepsilon^{2}(1+\varepsilon)^{2}}{1+2\varepsilon^{2}(1+\varepsilon)^{2}+\sqrt{1+4\varepsilon^{4}(1+\varepsilon)^{4}}}.
\end{displaymath}
Numerical behaviors of these bounds are shown in Figures~\ref{fig:comup2} and~\ref{fig:comlow2}.

\begin{table}[h!!]
\centering
\setlength{\tabcolsep}{11mm}{
\begin{tabular}{@{} cc @{}}
\toprule
\text{Estimate} & \text{Combined upper bound for $\mathscr{C}_{1}$} \\
\midrule
\eqref{Chen-comb1} & $\frac{1+2(1+\varepsilon)^{2}}{2(1+\varepsilon)^{2}}\big(1+2\varepsilon^{2}+2\varepsilon^{2}(1+\varepsilon)^{2}+\sqrt{1+4\varepsilon^{4}(1+\varepsilon)^{4}}\big)$ \\
\eqref{corup1.1} & $\frac{1}{2}+2\varepsilon^{2}+\varepsilon^{2}(1+\varepsilon)^{2}+\frac{1}{2}\sqrt{1+4\varepsilon^{4}(1+\varepsilon)^{4}}$ \\
\bottomrule
\end{tabular}}
\caption{\small The combined upper bounds in~\eqref{Chen-comb1} and~\eqref{corup1.1}.}
\label{tab:comup2.1}
\end{table}

\begin{table}[h!!]
\centering
\setlength{\tabcolsep}{2mm}{
\begin{tabular}{@{} cc @{}}
\toprule
\text{Estimate} & \text{Combined upper bound for $\mathscr{C}_{2}$} \\
\midrule
\eqref{Li-comb1} & $\frac{1+2(1+\varepsilon)^{2}}{(1+\varepsilon)^{2}}\big(1+2\varepsilon^{2}(1+\varepsilon)^{2}+\sqrt{1+4\varepsilon^{4}(1+\varepsilon)^{4}}\big)-2\varepsilon^{2}-(1+\varepsilon)^{2}$ \\
\eqref{corup1.2}  & $\frac{3}{2}+3\varepsilon^{2}(1+\varepsilon)^{2}+\frac{3}{2}\sqrt{1+4\varepsilon^{4}(1+\varepsilon)^{4}}$ \\
\bottomrule
\end{tabular}}
\caption{\small The combined upper bounds in~\eqref{Li-comb1} and~\eqref{corup1.2}.}
\label{tab:comup2.2}
\end{table}

\begin{figure}[h!!]
\centering
\begin{tabular}{cc}
\begin{minipage}[t]{3.1in}
\includegraphics[width=3.05in]{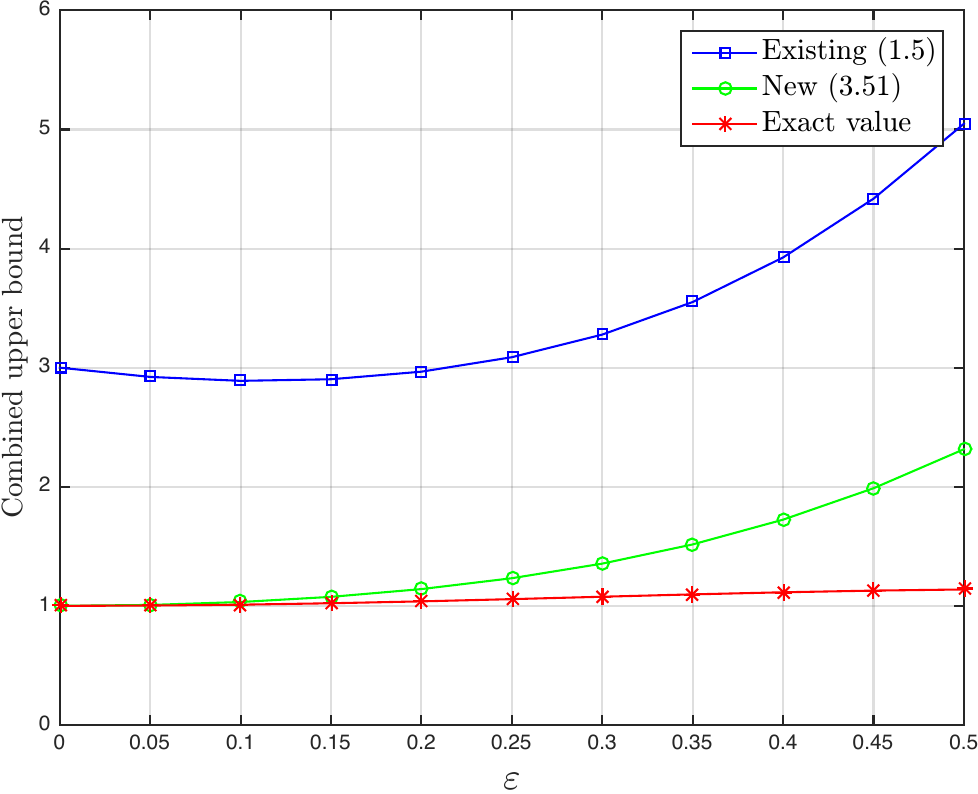}
\end{minipage}
\begin{minipage}[t]{3.1in}
\includegraphics[width=3.06in]{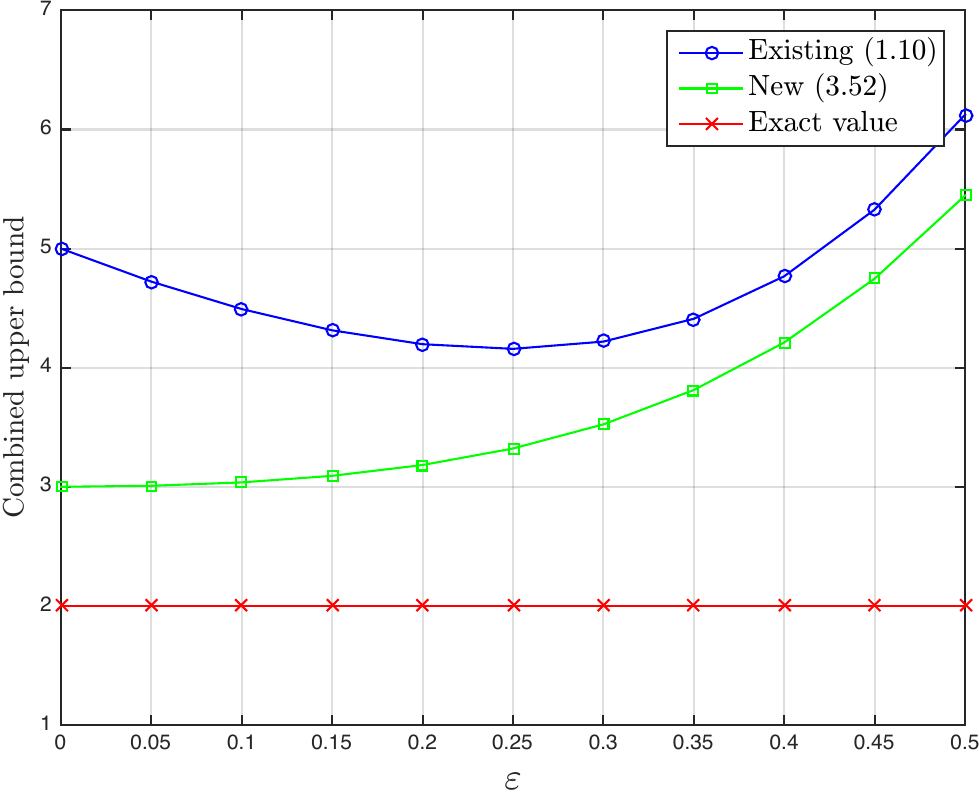}
\end{minipage}
\end{tabular}
\caption{\small Numerical comparison of the combined upper bounds for $\mathscr{C}_{1}$ (left); numerical comparison of the combined upper bounds for $\mathscr{C}_{2}$ (right).}
\label{fig:comup2}
\end{figure}

\begin{figure}[h!!]
\centering
\begin{tabular}{cc}
\begin{minipage}[t]{3.1in}
\includegraphics[width=3.05in]{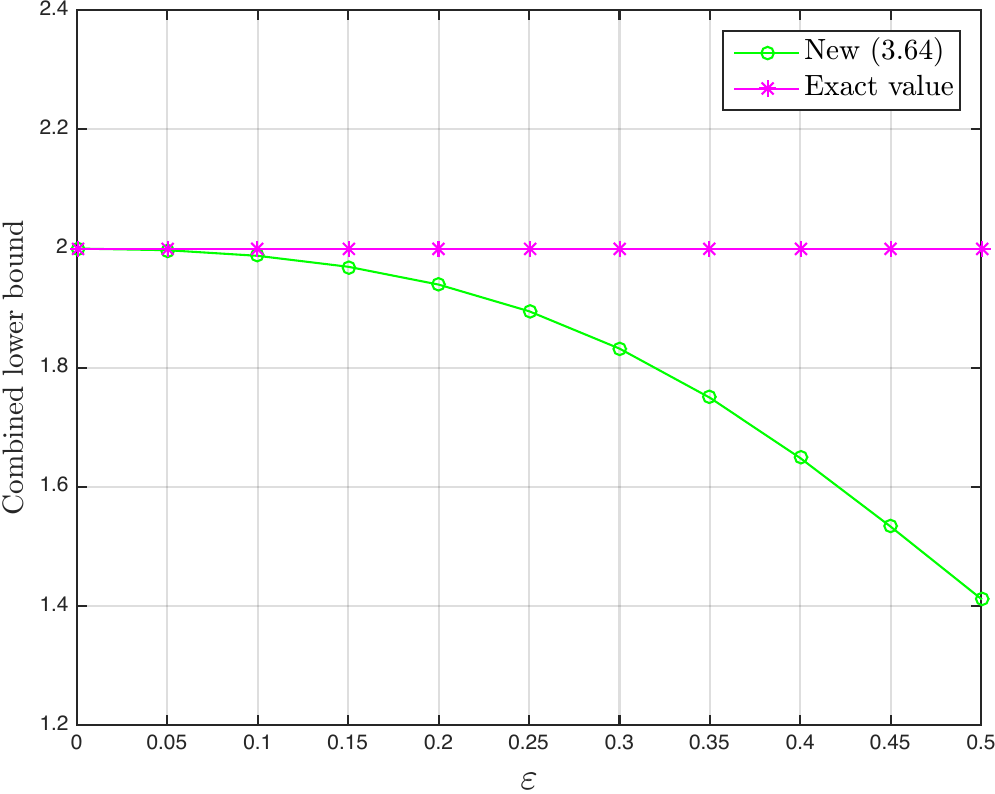}
\end{minipage}
\begin{minipage}[t]{3.1in}
\includegraphics[width=3.09in]{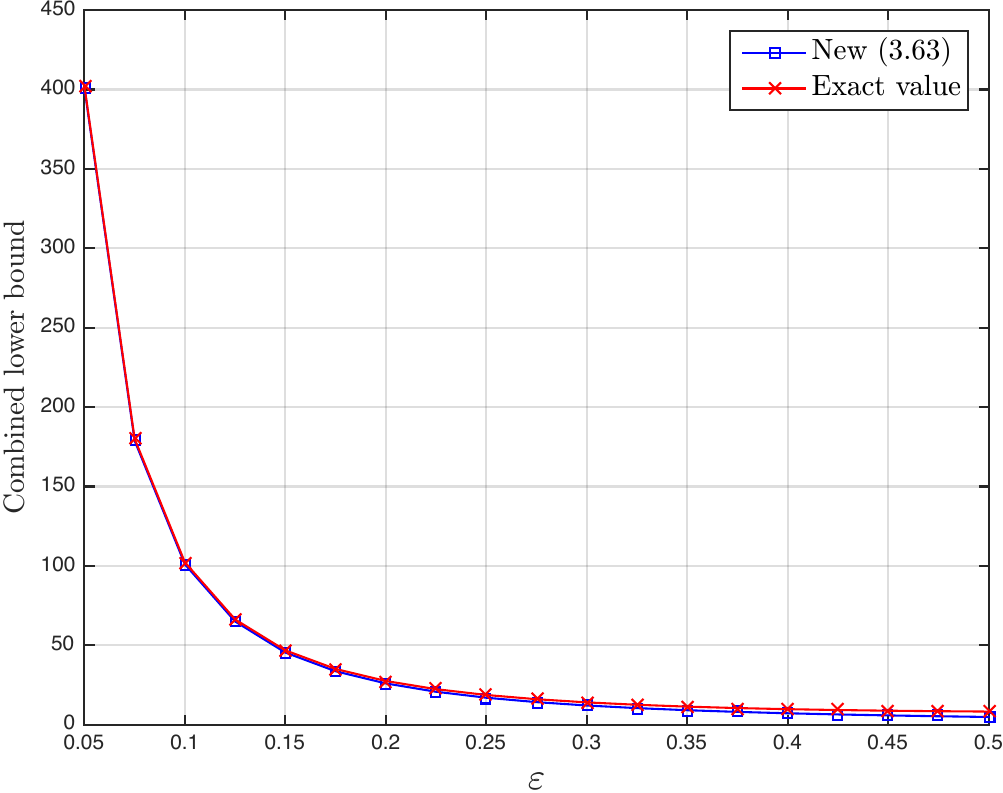}
\end{minipage}
\end{tabular}
\caption{\small Numerical behavior of the combined lower bound for $\mathscr{C}_{2}$ (left); numerical behavior ($\varepsilon$ is confined in $(0.05,0.5)$) of the combined lower bound for $\mathscr{C}_{3}$ (right).}
\label{fig:comlow2}
\end{figure}

Figure~\ref{fig:comup2} displays that the new combined upper bounds for $\mathscr{C}_{1}$ and $\mathscr{C}_{2}$ are smaller than the existing ones. Moreover, Figure~\ref{fig:comlow2} shows that the combined lower bounds in~\eqref{corlow1.1} and~\eqref{corlow1.2} are very close to the corresponding exact values (especially when $\varepsilon$ is small).

\section*{Acknowledgments}

The author would like to thank the anonymous referees for their valuable comments and suggestions, which greatly improved the original version of this paper. This research was carried out by the author during his Ph.D. study at the Academy of Mathematics and Systems Science, Chinese Academy of Sciences. The author is grateful to Professor Chen-Song Zhang for his kind support.

\bibliographystyle{abbrv}
\bibliography{references}

\end{document}